\documentclass[10pt]{article}

\usepackage[title,toc,titletoc,page]{appendix}

\usepackage{amsfonts}
\usepackage{amsmath}
\usepackage{amsthm}
\usepackage{amssymb}
\usepackage{mathtools}
\usepackage{multicol}
\usepackage{latexsym}
\usepackage[all]{xy}
\usepackage{xcolor}
\usepackage{multicol}
\usepackage{bm}
\usepackage{extarrows}

\usepackage{ytableau}

\newtheorem{theorem}{Theorem}[section]
\newtheorem{lemma}[theorem]{Lemma}
\newtheorem{proposition}[theorem]{Proposition}
\newtheorem{corollary}[theorem]{Corollary}
\newtheorem{definition}[theorem]{Definition}

\theoremstyle{remark}
\newtheorem{remark}[theorem]{Remark}

\newcommand{\eps}{\epsilon}

\newcommand{\id}{\mathrm{id}}

\newcommand{\mbA}{\mathbf{A}}

\newcommand{\mbJ}{\mathbf{J}}

\newcommand{\mbbF}{\mathbb{F}}

\newcommand{\mcA}{\mathcal{A}}

\newcommand{\mcN}{\mathcal{N}}

\newcommand{\mcP}{\mathcal{P}}

\newcommand{\mfB}{\mathfrak{B}}

\newcommand{\mfc}{\mathfrak{c}}

\newcommand{\mfgl}{\mathfrak{g}\mathfrak{l}}
\newcommand{\mfh}{\mathfrak{h}}

\newcommand{\mfp}{\mathfrak{p}}

\newcommand{\mft}{\mathfrak{t}}
\newcommand{\mfU}{\mathfrak{U}}

\newcommand{\msc}{\mathsf{c}}

\newcommand{\mse}{\mathsf{e}}
\newcommand{\msE}{\mathsf{E}}

\newcommand{\msK}{\mathsf{K}}

\newcommand{\msS}{\mathsf{S}}
\newcommand{\mss}{\mathsf{s}}
\newcommand{\mst}{\mathsf{t}}
\newcommand{\msT}{\mathsf{T}}

\newcommand{\msv}{\mathsf{v}}

\newcommand{\C}{\mathbb{C}}

\newcommand{\Z}{\mathbb{Z}}

\newcommand{\ot}{\otimes}

\newcommand{\ad}{\text{ad }}
\DeclareMathOperator{\sgn}{sgn}

\DeclareMathOperator{\ch}{ch}

\DeclareMathOperator{\rank}{rank}

\allowdisplaybreaks

\begin{document}

\setlength{\pdfpagewidth}{8.5in}
\setlength{\pdfpageheight}{11in}

\setlength{\parskip}{1em}

\begin{center}

\Large{\textbf{  Highest Weight Modules Over The Quantum Periplectic Superalgebra of Type $P$}}

Saber Ahmed, Dimitar Grantcharov, Nicolas Guay
\end{center}

\begin{abstract} 
In this paper, we begin the study of highest weight representations of the quantized enveloping superalgebra $\mfU_q\mfp_n$ of type $P$. We introduce a Drinfeld-Jimbo representation and establish a triangular-decomposition of $\mfU_q\mfp_n$. We explain how to relate modules over  $\mfU_q\mfp_n$ to modules over $\mfp_n$, the Lie superalgebra of type $P$, and we prove that the category of tensor modules over $\mfU_q\mfp_n$ is not semisimple.

{\small
}

\end{abstract}

\section*{Introduction} 

The classification of finite-dimensional simple Lie superalgebras and the foundations of their
representation theory was established by V. Kac in \cite{K1} and
\cite{K2}.  The representation theory of Lie
superalgebras has been known, since its inception, to be  more
complicated than that of  Lie algebras. The Lie superalgebras of types $P$ and $Q$ are especially interesting due to the algebraic, geometric, and combinatorial properties of their representations. The study of the representations of type $P$ Lie superalgebras, which are also called periplectic in the literature, has attracted considerable attention in the last several years. Interesting results on the category $\mathcal O$, the associated periplectic Brauer algebras, and related theories have been established in \cite{AGG}, \cite{BDEA^+1}, \cite{BDEA^+2}, \cite{CP}, \cite{Co}, \cite{CE1}, \cite{CE2}, \cite{DHIN}, \cite{EAS1}, \cite{EAS2}, \cite{HIR}, \cite{IN}, \cite{IRS}, \cite{KT}, \cite{Ser}, among others.

In this paper we initiate the study of highest weight representations of the quantum superalgebra $\mfU_q \mathfrak p_n$. In \cite{AGG} we constructed a flat deformation of the universal enveloping algebra $\mfU \mathfrak p_n$ which is a quantum enveloping superalgebra in the sense of Drinfeld (\cite{Dr}, \S 7). The idea was to apply a suitable modification of the procedure used by Faddeev, Reshetikhin, and Takhtajan in \cite{FRT}  using an element $S$ in $\mbox{End}(\C_q (n|n) ^{\otimes 2})$ that satisfies the quantum Yang-Baxter equation. 

In the present paper, based on the definition of  $\mfU_q \mathfrak p_n$ in \cite{AGG}, we give a presentation of $\mfU_q \mathfrak p_n$ in terms of Drinfeld-Jimbo generators and relations. These  relations are quantum deformations of those obtained in \cite{DKM}. Using this new presentation, we find a natural triangular decomposition of $\mfU_q \mathfrak p_n$, and then introduce the notion of highest weight module. This matches the corresponding result of Moon in \cite{M} for $\mfU \mathfrak p_n$. We also obtain the explicit decomposition of the second and the third tensor power of the natural representation of $\mfU_q \mathfrak p_n$. These decompositions, in particular, imply that the category of tensor representations is not semisimple, which is expected.

The structure of the paper is as follows. We give the notation and basic definitions related to the classical periplectic Lie superalgebra in Section \ref{sec:pn}. In Section \ref{sec:Uqpn}, we present a Drinfeld-Jimbo representation of $\mfU_q\mfp_n$ and prove its triangular decomposition. We introduce standard notation, definitions, and results related to highest weight $\mfU_q\mfp_n$-modules in Section \ref{sec:highest-weight-uqpn}. In Section \ref{sec:classical-limit} we discuss the classical limit and how the highest weight representations of $\mfU_q\mfp_n$ relate to those of $\mfp_n$ (cf. Theorem \ref{thm:classicallimit}). In the last section, we discuss tensor representations of $\mfU_q\mfp_n$ and use particular modules to prove that not every tensor representation of $\mfU_q\mfp_n$ is completely reducible (cf. Theorem \ref{thm:vk-not-reducible}).

\noindent \textbf{Acknowledgements:} The second named author is partly supported by the Simons Collaboration Grant 855678. The third named author gratefully acknowledges the financial support of the Natural Sciences and Engineering Research Council of Canada provided via the Discovery Grant Program. 

\section{The Lie superalgebra $\mfp_n$ and its representations}\label{sec:pn}

By $\mathbb Z_2 = \left\{0,1 \right\}$ we denote the group $\mathbb Z / 2 \mathbb Z$. All Lie superalgebras and homomorphisms are over $\C$ unless otherwise stated.

We will use the same setting as in \cite{AGG}. We will denote by $\C(n|n)$ as the vector superspace $\C^n\oplus \C^n$ spanned by the standard basis vectors $u_{-n},\hdots,u_{-1},u_1,\hdots,u_n$. We say that $u_i$ is odd if $i<0$ and even if $i>0$. Denote the elementary matrices in $M_{n|n}(\C)$, the vector superspace consisting of square $(2n)\times(2n)$-matrices with entries in $\C$ by $E_{ij}$, with $i,j\in \{\pm 1, \pm 2,\hdots, \pm n\}$. Set the parity function $p: \{\pm 1, \pm 2,\hdots, \pm n\} \longrightarrow \Z_2$ to be $p(i) = 0$ if $i>0$ and $p(i)=1$ if $i<0$. We set $\msE_{ij} = E_{ij} - (-1)^{p(i)(p(j)+1)}E_{-j,-i}$ and observe that $\msE_{ij} = - (-1)^{p(i)(p(j)+1)}\msE_{-j,-i}$ for all $i,j\in\{\pm 1, \ldots, \pm n\}$.  Therefore, $\msE_{i,-i} = 0$ when $1 \le i \le n$.

The Lie superalgebra $\mfp_n$ of type $P$ is the subsuperalgebra of $\mfgl(n|n)$ that consists of matrices of the form 
\begin{align*}
\begin{pmatrix}
A & B \\ C & D
\end{pmatrix}
\end{align*}
where $A,B,C,D\in \mfgl(n)$, $D = -A^t$, $B=B^t$, and $C=-C^t$. A basis of $\mfp_n$ is provided by all the matrices $\msE_{ij}$ with indices $i$ and $j$ respecting one of the following series of inequalities: $$1 \le |j| < |i| \le n \text{ or } 1\le i=j \le n \text{ or } -n \le i=-j \le -1.$$

The superbracket on $\mfp_n$ is given by \begin{eqnarray} [\msE_{ji}, \msE_{lk}] 
& = & \delta_{il} \msE_{jk} - (-1)^{(p(i)+p(j))(p(k)+p(l))} \delta_{jk} \msE_{li} \notag \\
& & - \delta_{i,-k}(-1)^{p(l)(p(k)+1)} \msE_{j,-l} - \delta_{-j,l}(-1)^{p(j)(p(i)+1)} \msE_{-i,k}. \label{supbr}
\end{eqnarray}

Throughout the paper $\mathfrak{h}$ will be the Lie subsuperalgebra of $\mathfrak{p}_n$ with basis $\{k_1,\hdots,k_n\}$, where $k_i \coloneqq \msE_{ii}$ for $1\leq i \leq n$. Note that $\mfh$ is purely even, and is also a self-normalizing nilpotent subsuperalgebra of $\mfp_n$, hence a Cartan subsuperalgebra of $\mathfrak{p}_n$. By $\{\epsilon_1,\hdots,\epsilon_n\}$ we denote the basis of $\mfh^*$ dual to $\{k_1,\hdots,k_n\}$.

Set $I \coloneqq \{1,\hdots, n-1\}$. The root system $\Delta$ of $\mfp_n$ relative to $\mathfrak h$ consists of the roots $\epsilon_i-\epsilon_j$ (for $i\neq j$), $\epsilon_i+\epsilon_j$ (for $i< j$), and $-\epsilon_i-\epsilon_j$ (for $i\leq j$). Let $\alpha_{i} = \epsilon_{i} -\epsilon_{i+1}$, $\beta_i = 2\epsilon_{i}$, and $\gamma_i = \epsilon_i + \epsilon_{i+1}$. Set

\begin{multicols}{3}
\noindent
\begin{align*}
e_i &\coloneqq \msE_{-i-1,-i},\\
f_i &\coloneqq \msE_{i+1,i},
\end{align*}
\begin{align*}
e_{\overline{i}} &\coloneqq \msE_{i+1,-i},\\
f_{\overline{i}} &\coloneqq \msE_{-i-1,i},
\end{align*}
\begin{align*}
F_{\overline{j}} &\coloneqq \msE_{-j,j}.
\end{align*}
\end{multicols}

\noindent for $i\in I$ and $j\in I \cup \{ n\}$. The root spaces of $\alpha_i$, $-\alpha_i$, $\gamma$, $-\gamma_i$, and $-\beta_i$ are spanned, respectively, by $e_i, f_i, e_{\overline{i}}, f_{\overline{i}}$, and $F_{\overline{i}}$. Note that $\beta_i \notin \Delta$.

Using the root space decomposition $\mfp_n = \mfh \oplus \left(\bigoplus\limits_{\mu\in\Delta} (\mfp_n)_\mu\right)$ we define the triangular decomposition $\mfp_n = \mfp_n^- \oplus \mfh \oplus \mfp_n^+$ as follows: $\mfp_n^-$ is spanned by $\{f_i, f_{\overline{i}}, F_{\overline{j}} \mid i\in I, j\in I \cup \{ n\}\}$ and $\mfp_n^+$ is spanned by $\{e_i, e_{\overline{i}} \mid i\in I\}$.  Alternatively, $\Delta = \Delta_+ \sqcup \Delta_-$, where \begin{eqnarray*} \Delta_+ =  \Delta(\mfp_n^+) = \{\alpha_i,\gamma_i \mid i\in I\}, & \Delta_- = \Delta(\mfp_n^-) = \{-\alpha_i, -\gamma_i, -\beta_j \mid i\in I, j\in I \cup \{ n\}\}.
\end{eqnarray*}
In this paper, all highest weight modules of $\mfp_n$  will be relative to the Borel subalgebra $\mathfrak  b_n = \mfh \oplus \mfp_n^+$.

The cone of positive roots will be denoted by $\displaystyle Q_+ \coloneqq \sum\limits_{i=1}^{n-1} \Z_{\geq 0}\alpha_i + \sum\limits_{i=1}^{n-1} \Z_{\geq 0}\gamma_i$ and $\displaystyle Q_- \coloneqq -\sum \limits_{i=1}^{n-1} \Z_{\geq 0}\alpha_i - \sum \limits_{i=1}^{n-1} \Z_{\geq 0}\gamma_i - \sum \limits_{i=1}^n \Z_{\geq 0}\beta_i$ denotes the cone of negative roots.  Set $Q= Q_+  + Q_-$.

We will also denote $\displaystyle P \coloneqq \bigoplus\limits_{i=1}^n \Z\epsilon_i$ to be the weight lattice of $\mfp_n$, and denote $P^\vee \coloneqq \bigoplus\limits_{i=1}^n \Z k_i$ to be the coweight lattice.

We next give a presentation of $\mfp_n$ (hence of $\mfU\mfp_n$) in terms of generators and relations. This presentation will be used to define $\mfU_q\mfp_n$ in terms of Drinfeld-Jimbo generators and relations. 

{\allowdisplaybreaks
\begin{proposition}[\cite{DKM}, Definition 3.1.1] \label{prop:pn} 
The complex Lie superalgebra $\mfp_n$ is generated by the elements $e_i$, $e_{\overline{i}}$, $f_i$, $f_{\overline{i}}$ $(i\in I)$, $\mfh$ and $F_{\overline{j}}$ ($j\in I \cup \{ n\}$) subject to following defining relations (for $h\in\mfh$):

\begin{multicols}{2}
\noindent
\begin{align*}
&[\mfh,\mfh] = 0\\
&[h,e_i] = \alpha_i(h)e_i\\
&[h,f_i] = -\alpha_i(h)f_i\\
&[h,e_{\overline{i}}] = \gamma_i(h)e_{\overline{i}}\\
&[h,f_{\overline{i}}] = -\gamma_i(h)f_{\overline{i}}\\
&[h,F_{\overline{i}}] = -\beta_i(h)F_{\overline{i}}\\
&[e_i,e_j] = [f_i,f_j] = 0 \text{ for } |i-j| \neq 1 \\
&[e_i,f_j] = -\delta_{ij}(k_i-k_{i+1})\\
&[e_{\overline{i}},f_{\overline{i}}] = -(k_i-k_{i+1})\\
&[f_{\overline{i}},e_{\overline{j}}] = 0 \text{ if } |i-j| > 1\\
&[f_{\overline{i+1}},e_{\overline{i}}] = [e_{i+1},e_i] \\
&[f_{\overline{i}},e_{\overline{i+1}}] = [f_{i+1},f_i]\\
&[e_{\overline{i+1}},e_i] = [f_{i+1},e_{\overline{i}}]\\
&[f_{\overline{i+1}},f_i] = [e_{i+1},f_{\overline{i}}]\\
&[f_{\overline{i}},f_{i}] = F_{\overline{i}} \\
&[e_i,f_{\overline{i}}] = F_{\overline{i+1}} \\
&[e_{\overline{i}},e_{\overline{j}}] = [f_{\overline{i}},f_{\overline{j}}] = 0 \text{ for } i,j\in I \\
& [f_i,e_{\overline{j}}] = 0 \text{ if } i\neq j+1\\
& [e_{\overline{i}}, e_j] = 0 \text{ if } i\neq j+1 \\
& [e_i, f_{\overline{j}}] = 0 \text{ if } i\neq j,j+1 \\
& [f_{\overline{i}}, f_j] = 0 \text{ if } i\neq j,j+1 \\
&[F_{\overline{j}}, e_i] = -\beta_i(k_j)f_{\overline{i}}\\
&[F_{\overline{j}}, f_i] = \beta_{i+1}(k_j)f_{\overline{i}} \\
&[e_i,[e_i,e_{i\pm1}]] = 0 \\
&[f_i,[f_i,f_{i\pm1}]] = 0 \\
&[e_{\overline{i+1}},[e_{i+1},e_{i}]] = e_{\overline{i}} \\
\end{align*}
\end{multicols}
\end{proposition}
}

\begin{remark}
We note that we use a slightly different presentation of $\mfp_n$ in terms of generators and relations than the one used in  Definition 3.1.1 in \cite{DKM}. To define an isomorphism between the two presentations we proceed as follows.
A homomorphism from the presentation in Proposition \ref{prop:pn} to the one in  Definition 3.1.1 in \cite{DKM} can be defined by the following maps:
\begin{align*}
k_i \longmapsto -H_i, &\hspace{20pt} e_i \longmapsto -F_i, \hspace{20pt} f_i \longmapsto E_i,\\ f_{\overline{i}} \longmapsto B_i, &\hspace{20pt} e_{\overline{i}} \longmapsto -C_i, \hspace{20pt} F_{\overline{1}} \longmapsto -\frac{1}{2} B_{1,1}.
\end{align*}
These maps indeed define a homomorphism because $F_{\overline{j}} = [e_{j-1}, f_{\overline{j-1}}] = [ f_{\overline{j}}, f_j]$ and all relations listed in Proposition \ref{prop:pn} follow from the relations in Definition 3.1.1 (see for example Lemma 3.2.1 in \cite{DKM}).
The details are left to the reader.  To define a reverse homomorphism is easier. We note that neither of the sets of generators is minimal, but the larger set of generators used in this paper will serve better our purpose. 
\end{remark}

The following are relations of $\mfp_n$ that can be obtained from the relations in Proposition \ref{prop:pn}.

\begin{lemma} The following relations hold in $\mfp_n$:

\begin{itemize}
\item[(a)] $\displaystyle [F_{\overline{j}}, e_{\overline{i}}] = \begin{cases} 2f_i & \text{ if } j = i\\ 2e_i& \text{ if } j = i +1 \\ 0 & \text{ otherwise} \end{cases}$,
\item[(b)] $\displaystyle  [F_{\overline{j}}, f_{\overline{i}}] = 0$,
\item[(c)] $\displaystyle [e_i,[e_i,e_{\overline{i\pm1}}]] = 0$,
\item[(d)] $\displaystyle [f_i,[f_i,f_{\overline{i\pm1}}]] = 0$,
\item[(e)] $\displaystyle [F_{\overline{i}},F_{\overline{j}}] = 0 \text{ for } i,j\in I \cup \{ n\}$.
\end{itemize}
\end{lemma}

\begin{proof} We will prove  (a) and (c). The remaining parts can be deduced similarly.

First, we prove (a) for $j=n$. For every $i$ we have
\begin{align*}
[F_{\overline{n}}, e_{\overline{i}}] &= [[e_{n-1},f_{\overline{n-1}}],e_{\overline{i}}] = [e_{\overline{i}},[e_{n-1},f_{\overline{n-1}}]] = [e_{n-1}, [f_{\overline{n-1}},e_{\overline{i}}]] +  [f_{\overline{n-1}}, [e_{\overline{i}},e_{n-1}]] = [e_{n-1}, [f_{\overline{n-1}},e_{\overline{i}}]].
\end{align*}
If $i = n-2$, then we have that $[e_{n-1}, [f_{\overline{n-1}},e_{\overline{n-2}}]] = [e_{n-1}, [e_{n-1},e_{n-2}]] = 0$. If $i = n-1$, then we have that $$[e_{n-1}, [f_{\overline{n-1}},e_{\overline{n-1}}]] = [e_{n-1}, -k_{n-1}+k_n] = -[e_{n-1},k_{n-1}] + [e_{n-1},k_n] = 2e_{n-1}.$$
Otherwise, we have that $[F_{\overline{n}}, e_{\overline{i}}] = 0$.

Next we prove (a) for $j < n$. Using the relations in Proposition \ref{prop:pn}, we have that:
\begin{align*}
[F_{\overline{j}}, e_{\overline{i}}] &= [[f_{\overline{j}},f_j],e_{\overline{i}}]\\
&= [e_{\overline{i}},[f_{\overline{j}},f_j]]\\
&=   [f_{\overline{j}},[f_j, e_{\overline{i}}]] - [f_j, [e_{\overline{i}}, f_{\overline{j}}]].
\end{align*}
Note that $[F_{\overline{j}}, e_{\overline{i}}] = 0$ from above, unless $|i-j| \leq 2$. So, we need to check the three subcases $i-j = 0,1,-1$.

If $j=i$, then
\begin{align*}
[F_{\overline{i}}, e_{\overline{i}}] &=  [f_{\overline{i}},[f_i, e_{\overline{i}}]] - [f_i, [e_{\overline{i}}, f_{\overline{i}}]]\\
&=  -[f_i,  -k_i + k_{i+1}]\\
 &= 2f_i.
\end{align*}
If $i = j+1$, then
\begin{align*}
[F_{\overline{j}}, e_{\overline{j+1}}]  &= [f_{\overline{j}},[f_j, e_{\overline{j+1}}]] - [f_j, [e_{\overline{j+1}}, f_{\overline{j}}]]\\  &=  - [f_j, [f_{j+1}, f_j]] \\ &=  [f_j, [f_j, f_{j+1}]] + [f_j, [f_{j+1}, f_j]]  \\  &= 0.
\end{align*}

If $j=i+1$, then
\begin{align*}
[F_{\overline{i+1}}, e_{\overline{i}}]  &=   [f_{\overline{i+1}},[f_{i+1}, e_{\overline{i}}]] - [f_{i+1}, [e_{\overline{i}}, f_{\overline{i+1}}]]\\
 &= [f_{\overline{i+1}},[e_{\overline{i+1}}, e_{i}]] - [f_{i+1}, [e_{i+1},e_i]] \\
  &= [e_{\overline{i+1}}, [e_{i},f_{\overline{i+1}}]] - [e_i, [f_{\overline{i+1}},e_{\overline{i+1}}]] +  [e_{i+1}, [e_i,f_{i+1}]] + [e_i, [f_{i+1}, e_{i+1}]]  \\
   &= - [e_i, -k_{i+1} + k_{i+2}] - [e_i, -k_{i+1}+k_{i+2}]\\  
   &= 2e_i.
\end{align*}

Now, we prove (c). Note that $[e_i,e_{\overline{i-1}}] = 0$ for all $2\leq i \leq n$, so $[e_i, [e_i,e_{\overline{i-1}}]] = 0$. Also,
\begin{align*}
[e_i,[e_i,e_{\overline{i+1}}]] &= [e_i,[e_{\overline{i}}, f_{i+1} ]]\\
&= [e_{\overline{i}}, [f_{i+1}, e_i]] + [f_{i+1}, [e_i, e_{\overline{i}}]]\\
&=  0.
\end{align*}
\end{proof}

\section{Quantized enveloping superalgebra $\mfU_q\mfp_n$}\label{sec:Uqpn}

Let $\C(q)$ be the field of rational functions in the variable $q$, and let $\C_q(n|n) = \C(q) \otimes_{\C} \C(n|n)$. Definition 3.6 from \cite{AGG} gives that $\mfU_q\mfp_n$ is defined to be the associative superalgebra over $\C(q)$ generated by elements $t_{ij},t_{ii}^{-1}$ with $1 \le |i|\le |j|\le n$ and $i,j\in \{ \pm 1,\ldots, \pm n  \}$, such that $t_{ii} = t_{-i,-i}$, $t_{-i,i} = 0$ if $i\geq 0$, $t_{ij} = 0$ if $|i| > |j|$, and the following relation is satisfied:
\begin{equation} 
\begin{split}  
& (-1)^{(p(i)+p(j))(p(k)+p(l))} t_{ij}t_{kl} - t_{kl}t_{ij} + \theta(i,j,k) \big( \delta_{|j|<|l|} - \delta_{|k|<|i|} \big) \eps t_{il}t_{kj} \\
& \qquad + (-1)^{(p(i)+p(j))(p(k)+p(l))} \big( \delta_{j>0}(q-1) + \delta_{j<0}(q^{-1}-1) \big) \big(\delta_{jl} + \delta_{j,-l} \big) t_{ij}t_{kl} \\
& \qquad  \qquad  - \big( \delta_{i>0}(q-1) + \delta_{i<0}(q^{-1}-1) \big) \big( \delta_{ik} + \delta_{i,-k} \big) t_{kl}t_{ij} \label{exprel} \\
&  \qquad  \qquad  \qquad + \theta(i,j,k) \delta_{j>0}  \delta_{j,-l}  \eps t_{i,-j}t_{k,-l} - (-1)^{p(j)} \delta_{i<0} \delta_{i,-k}  \eps t_{-k,l}t_{-i,j} \\
 & \qquad  + (-1)^{p(j)(p(i)+1)} \eps \sum_{-n \le a \le n} \big( (-1)^{p(i)p(a)}\theta(i,j,k) \delta_{j,-l} \delta_{|a|<|l|} t_{i,-a}t_{ka} + (-1)^{p(-j)p(a)}\delta_{i,-k} \delta_{|k|<|a|} t_{al}t_{-a,j}\big) \\
& \qquad  = 0,  
\end{split} 
\end{equation}
where $\epsilon = q-q^{-1}$ and $\theta(i,j,k) = \sgn(\sgn(i) +\sgn(j) + \sgn(k))$.

Now, let 
\begin{align}
\begin{split}
q^{k_i} \coloneqq t_{ii}, \hspace{20pt}  e_i \coloneqq \frac{-1}{q-q^{-1}}t_{-i,-i-1}, \hspace{20pt} f_{\overline{i}} \coloneqq  \frac{-1}{q-q^{-1}}t_{i,-i-1}, \\  f_i \coloneqq  \frac{1}{q-q^{-1}}t_{i,i+1}, \hspace{20pt}  e_{\overline{i}} \coloneqq  \frac{1}{q-q^{-1}}t_{-i,i+1}, \hspace{20pt} F_{\overline{i}} \coloneqq \frac{-2}{q-q^{-1}}t_{i,-i}. \label{Uqpngen} 
\end{split}
\end{align}

Using (\ref{Uqpngen}), we have the following relations between the two sets of generators of $\mfU_q\mfp_n$
\begin{equation}\label{qB2}
\begin{aligned}
\displaystyle t_{-i,-i-j} &= -(q-q^{-1})q^{-\sum_{h=1}^{j-1}k_{i+h}}\prod\limits_{h=1}^{j-1} \ad e_{i+h}(e_i), \\
\displaystyle t_{-i,i+j} &= (q-q^{-1})q^{-\sum_{h=1}^{j-1}k_{i+h}}\prod\limits_{h=1}^{j-1} \ad f_{i+h}(e_{\overline{i}}), \\
\displaystyle t_{i,-i-j} &= -(q-q^{-1})q^{-\sum_{h=1}^{j-1}k_{i+h}}\prod\limits_{h=1}^{j-1} \ad e_{i+h}(f_{\overline{i}}), \\
\displaystyle t_{i,i+j} &= (q-q^{-1})q^{-\sum_{h=1}^{j-1}k_{i+h}}\prod\limits_{h=1}^{j-1} \ad f_{i+h}(f_{i}),
\end{aligned} 
\end{equation}
where $\ad a_i(a_j)\coloneqq [a_i,a_j]$, $\displaystyle \prod\limits_{h=1}^{j}\ad a_{i+h}(a_i) \coloneqq \ad a_{i+j} \ad a_{i+j-1} \hdots \ad a_{i+1}(a_i)$, and $\displaystyle \prod\limits_{h=1}^{0}\ad a_{i+h}(a_i) \coloneqq a_i$, for $a_i = e_i, e_{\overline{i}}, f_i, f_{\overline{i}}$.  From \eqref{qB2}, one can  obtain the following relations
\begin{equation}\label{eq-induction-1}
\begin{aligned} 
t_{i,i+j} &= q^{-k_{i+j-1}}(f_{i+j-1}t_{i,i+j-1} - t_{i,i+j-1}f_{i+j-1}),\\
t_{-i,i+j} &= q^{-k_{i+j-1}}(f_{i+j-1}t_{-i,i+j-1} - t_{-i,i+j-1}f_{i+j-1}), \\
t_{i,-i-j} &= q^{-k_{i+j-1}}(e_{i+j-1}t_{i,-i-j+1}-t_{i,-i-j+1}e_{i+j-1}),\\
t_{-i,-i-j} &= q^{-k_{i+j-1}}(e_{i+j-1}t_{-i,-i-j+1}-t_{-i,-i-j+1}e_{i+j-1}).
\end{aligned}
\end{equation}

Equivalently, the above relations can be written as follows
\begin{equation}\label{eq-induction-2}
\begin{aligned} 
t_{ij} &= -q^{-k_{i+1}}(f_{i}t_{i+1,j} - t_{i+1,j}f_{i}),\\
t_{-i,j} &= q^{-k_{i+1}}(e_{i}t_{i+1,j} - t_{i+1,j}e_{i}),\\
\end{aligned}
\end{equation}
where $i>0$. 

The relations  \eqref{eq-induction-1} (respectively, \eqref{eq-induction-2}) can serve as an alternative way to define the generators of $\mfU_q\mfp_n$ inductively. These relations also allows us to obtain the following relations.
\begin{lemma} \label{lem:alg}
The following relations hold in $\mfU_q\mfp_n$ for all $i \in I$.
\begin{itemize}
\item[(a)] $e_if_i - f_ie_i = e_{\overline{i}}f_{\overline{i}} + f_{\overline{i}} e_{\overline{i}}$,
\item[(b)] $\displaystyle \frac{2}{1+q^2}f_{\overline{i+1}}f_{i+1}f_i - f_{\overline{i+1}}f_if_{i+1} - f_{i+1}f_if_{\overline{i+1}} + q^2f_if_{i+1}f_{\overline{i+1}} = q^2q^{2k_{i+1}}f_{\overline{i}} - \frac{1-q^2}{1+q^2}f_{i+1}f_{\overline{i+1}}f_i$,
\item[(c)] $\displaystyle f_ie_i = e_if_i+q^2\frac{q^{2k_i}-q^{k_{i+1}}}{q^2-1} + (q^2-1)e_{\overline{i}}f_{\overline{i}}$.
\end{itemize}
\end{lemma}

Our first main result is the following presentation of $\mfU_q\mfp_n$.

\begin{proposition} \label{prop:alg} The quantum superalgebra $\mfU_q\mfp_n$ is isomorphic to the unital associative superalgebra over $\C(q)$ generated by the even elements  $q^h$ for $h\in P^{\vee}$, $e_i, f_i$ for $i\in I$, and the odd elements $e_{\overline{i}}, f_{\overline{i}}$, for $i\in I$,  $F_{\overline{i}}$ for $i\in I \cup \{ n\} $, that satisfy the following relations

{\allowdisplaybreaks
\begin{align*}
& q^0 = 1,  \;\; q^{h_1+h_2} = q^{h_1}q^{h_2}  \;\; \text{ for } h_1, h_2\in  P^{\vee}, \\
& q^he_i = q^{\alpha_i(h)}e_iq^{h}, \;\; q^hf_i = q^{-\alpha_i(h)}f_iq^{h} \;\; \text{ for } h\in  P^{\vee},\\
& q^he_{\overline{i}} = q^{\gamma_i(h)}e_{\overline{i}}q^{h}, \;\; q^hf_{\overline{i}} = q^{-\gamma_i(h)},
f_{\overline{i}}q^{h},\;\; q^hF_{\overline{i}} = q^{-\beta_i(h)}F_{\overline{i}}q^{h} \;\; \text{ for } h\in  P^{\vee}, \\ \\
& e_ie_j - e_je_i = 0,  \;\; f_if_j - f_jf_i = 0,  \;\; f_{\overline{i}}f_{\overline{j}} + f_{\overline{j}}f_{\overline{i}}= 0  \;\; \text{ if } |i - j| > 1,\\
  & e_{\overline{i}}e_{\overline{j}}+e_{\overline{j}}e_{\overline{i}} = 0, \;\; F_{\overline{i}}F_{\overline{j}}+F_{\overline{j}}F_{\overline{i}} = 0 \;\; \text{ if } |i - j| > 0, \\
& e_if_j - f_je_i = 0  \;\; \text{ if } j \neq i, i+1,\\
&  e_{i}f_{\overline{j}}-f_{\overline{j}}e_{i} = 0,  \;\; f_{i}f_{\overline{j}}-f_{\overline{j}}f_{i} = 0,  \;\; e_{\overline{i}} f_{\overline{j}}+f_{\overline{j}}e_{\overline{i}} = 0  \;\; \text{  if } |i - j| > 1,\\
&    e_{{i}}e_{\overline{j}}-e_{\overline{j}}e_{{i}} = 0,  \;\; f_{j}e_{\overline{i}}-e_{\overline{i}}f_{j} = 0 \;\; \text{  if } j \neq i+1,     \\ \\
  & F_{\overline{i}}e_j - e_{j}F_{\overline{i}} = 0,  \;\; F_{\overline{i}}f_j - f_{j}F_{\overline{i}} = 0  \;\; \text{ if } i\neq j, j+1,\\
    &  F_{\overline{i}}e_{\overline{j}}+e_{\overline{j}}F_{\overline{i}} = 0,  \;\;  F_{\overline{i}}f_{\overline{j}}+f_{\overline{j}}F_{\overline{i}} = 0  \;\; \text{  if } i \neq j, j+1, \\
   & e_{\overline{i}}^2 = 0,  \;\; f_{\overline{i}}^2 = 0, \;\; F_{\overline{i}}^2 = 0,    \\ \\
& e_{i+1}e_i -  e_ie_{i+1} = e_{\overline{i}}f_{\overline{i+1}} + f_{\overline{i+1}}e_{\overline{i}},  \;\;  f_{i+1}f_i- f_if_{i+1} = f_{\overline{i}}e_{\overline{i+1}} + e_{\overline{i+1}}f_{\overline{i}},\\
 & e_{\overline{i+1}}e_i - e_ie_{\overline{i+1}} = f_{i+1}e_{\overline{i}} - e_{\overline{i}}f_{i+1}, \;\; \displaystyle f_{\overline{i+1}}f_i - f_if_{\overline{i+1}} = e_{i+1}f_{\overline{i}} - f_{\overline{i}}e_{i+1}, \\
 & e_if_i - f_ie_i = -\frac{q^{2k_i} -q^{2k_{i+1}}}{q^2-1} + \frac{q^2-1}{q^2}f_{\overline{i}}e_{\overline{i}},\\
&  f_{\overline{i}}e_{\overline{i}} + q^2e_{\overline{i}}f_{\overline{i}} = -\frac{q^2}{q^2-1}(q^{2k_i} - q^{2k_{i+1}}),\\
 & qe_if_{\overline{i}} - q^{-1}f_{\overline{i}}e_i = \frac{(1+q^2)}{2}q^{k_{i+1}}F_{\overline{i+1}} = \displaystyle q^{-1}f_{\overline{i+1}}f_{i+1} - qf_{i+1}f_{\overline{i+1}},    \\ \\
& qF_{\overline{i+1}}e_i - e_{i}F_{\overline{i+1}} = 0,  \;\; qF_{\overline{i}}f_i - f_{i}F_{\overline{i}} = 0,\\
& F_{\overline{i}}e_i - qe_{i}F_{\overline{i}} = -2f_{\overline{i}}q^{k_i},  \;\; q^{-1}F_{\overline{i+1}}f_i - f_{i}F_{\overline{i+1}} = 2q^{k_{i+1}}f_{\overline{i}},\\
& F_{\overline{i}}e_{\overline{i}} + qe_{\overline{i}}F_{\overline{i}} = 2f_iq^{k_i},  \;\; F_{\overline{i}}f_{\overline{i}} + q^{-1}f_{\overline{i}}F_{\overline{i}} = 0, \\
& F_{\overline{i+1}}e_{\overline{i}} + qe_{\overline{i}}F_{\overline{i+1}} = 2e_iq^{k_{i+1}},  \;\; F_{\overline{i+1}}f_{\overline{i}} + q^{-1}f_{\overline{i}}F_{\overline{i+1}} = 0,     
\\ \\
& q^{-1}e_i^2e_{i+1} - (q+q^{-1})e_ie_{i+1}e_i + qe_{i+1}e_i^2 = 0,\\
 & qe_{i+2}^2e_{i} - (q+q^{-1})e_{i+1}e_ie_{i+1} + q^{-1}e_ie_{i+1}^2 = 0,\\
& qf_i^2f_{i+1} - (q+q^{-1})f_if_{i+1}f_i + q^{-1}f_{i+1}f_i^2 =  0,\\
& q^{-1}f_{i+1}^2f_i - (q+q^{-1})f_{i+1}f_if_{i+1} + qf_if_{i+1}^2 = 0,\\
& q^{-1}e_i^2e_{\overline{i+1}} - (q+q^{-1})e_ie_{\overline{i+1}}e_i + qe_{\overline{i+1}}e_i^2 = 0,\\
& qf_i^2f_{\overline{i+1}} - (q+q^{-1})f_if_{\overline{i+1}}f_i + q^{-1}f_{\overline{i+1}}f_i^2 = 0,\\
& e_{i+1}e_ie_{\overline{i+1}} - e_ie_{i+1}e_{\overline{i+1}} - q^2e_{\overline{i+1}}e_{i+1}e_i +q^{2}e_{\overline{i+1}} e_ie_{i+1}=q^{2k_{i+1}}e_{\overline{i}},    \\ \\
& { 2qq^{k_{i+1}}(f_{i+1}f_{\overline{i}} - f_{\overline{i}}f_{i+1}) = (1-q^{-2})F_{\overline{i+1}}(f_{i+1}f_i - f_if_{i+1})}, \\
& { -2qq^{k_{i+1}}(f_{\overline{i+1}}e_{i} - e_{i}f_{\overline{i+1}}) = (1-q^{-2})F_{\overline{i+1}}(e_{i+1}e_i - e_ie_{i+1})}, \\
& { -2qq^{k_{i+1}}(f_{\overline{i+1}}f_{\overline{i}}  + f_{\overline{i}} f_{\overline{i+1}}) = (1-q^{-2})F_{\overline{i+1}}(f_{\overline{i+1}}f_i - f_if_{\overline{i+1}})}, \\
& { 2qq^{k_{i+1}}(f_{i+1}e_{i}  - e_{i} f_{i+1}) = (1-q^{-2})F_{\overline{i+1}}(e_{\overline{i+1}}e_i - e_ie_{\overline{i+1}})}, \\
\end{align*}
}
\end{proposition}

\begin{proof}
Let $U$ be the unital associative superalgebra over $\C(q)$ generated by the elements $e_i, f_i, e_{\overline{i}}, f_{\overline{i}}$ for $i\in I$, $F_{\overline{i}}$ for $i\in I \cup \{ n\}$, and $q^h$ for $h\in P^{\vee}$ with defining relations given in the statement of the proposition above. 

We first note that using  (\ref{Uqpngen}) and particular choices for $i,j,k,l$ in 
 (\ref{exprel})   one can establish on a case-by-case basis all relations in the proposition. Hence, we have an associative superalgebra homomorphism $\psi: U\rightarrow \mfU_q\mfp_n$. The relations in (\ref{qB2}) immediately show that $\psi$ is surjective. 

It remains to show that $\psi$ is injective. For this, we prove that (\ref{exprel}) is obtained from the relations in the statement of the proposition by  considering the following 26 cases:

\begin{multicols}{4}
    \begin{enumerate}
        \item $\vert i \vert = \vert j \vert < \vert k \vert < \vert l \vert$
        \item $\vert k \vert < \vert i \vert = \vert j \vert < \vert l \vert$
        \item $\vert k \vert  < \vert l \vert < \vert i \vert = \vert j \vert $
        \item $\vert k \vert = \vert l \vert < \vert i \vert <  \vert j \vert $
        \item $\vert i \vert < \vert k \vert = \vert l \vert <\vert j \vert$
        \item $\vert i \vert <  \vert j \vert < \vert k \vert = \vert l \vert$
        \item $\vert i \vert = \vert k \vert < \vert j \vert < \vert l \vert$
        \item $\vert i \vert = \vert k \vert <\vert l \vert < \vert j \vert$
        \item $\vert i \vert < \vert k \vert <\vert l \vert = \vert j \vert$
        \item $\vert k \vert < \vert i \vert <\vert l \vert = \vert j \vert$
        \item $\vert k \vert < \vert l \vert =\vert i \vert < \vert j \vert$
        \item $\vert i \vert < \vert j \vert =\vert k \vert < \vert l \vert$
        \item $\vert i \vert < \vert j \vert < \vert k \vert < \vert l \vert$
        \item $\vert i \vert < \vert k \vert <\vert j \vert < \vert l \vert$
        \item $\vert i \vert < \vert k \vert <\vert l \vert < \vert j \vert$
        \item $\vert k \vert < \vert i \vert <\vert j \vert < \vert l \vert$
        \item $\vert k \vert < \vert i \vert <\vert l \vert < \vert j \vert$
        \item $\vert k \vert < \vert l \vert <\vert i \vert < \vert j \vert$
        \item $\vert k \vert = \vert l \vert <\vert i \vert = \vert j \vert$
        \item $\vert i \vert = \vert j \vert < \vert k \vert = \vert l \vert$
        \item $\vert i \vert = \vert k \vert < \vert j \vert = \vert l \vert$
        \item $\vert i \vert = \vert j \vert = \vert k \vert < \vert l \vert$
        \item $\vert k \vert < \vert i \vert = \vert j \vert = \vert l \vert$
        \item $\vert i \vert = \vert k \vert = \vert l \vert < \vert j \vert$
        \item $\vert i \vert < \vert j \vert = \vert k \vert = \vert l \vert$
        \item $\vert i \vert = \vert j \vert = \vert k \vert = \vert l \vert$
    \end{enumerate}
    \end{multicols}
    
The verification in each case uses the relations \eqref{eq-induction-1} and \eqref{eq-induction-2} and appropriate induction. In fact, for some cases, we apply useful  identities that follow from  \eqref{eq-induction-1} and \eqref{eq-induction-2}, see Lemma \ref{lem:alg} below. For example, in Case 21 we use Lemma \ref{lem:alg}(c). For reader's convenience, we write detailed proofs for cases 2 and 25. The remaining cases are established using analogous reasoning. 

\noindent{\it Case 2.} Suppose that $\vert k \vert < \vert i \vert = \vert j \vert < \vert l \vert$ in (\ref{exprel}). We prove that 
\begin{align*}
 (-1)^{(p(i)+p(j))(p(k)+p(\ell))}t_{ij}t_{k\ell} - t_{k\ell}t_{ij} = 0
\end{align*}
is obtained from some of the relations in the proposition applying  induction on  $|\ell| - |i|$ first, and  then induction on $|i| - |k|$. We consider only the case of when $i=-j >0$ and $k,\ell > 0$ as the other cases follow similarly. 

We start with the base case of the first induction, i.e.,  $\ell = i+1$. For the base case of the second induction, we have $i = k+1$. Then
\begin{small}
\begin{align*}
t_{k+1,-k-1}t_{k,k+2} - t_{k,k+2}t_{k+1,-k-1} &= -\frac{(q-q^{-1})^2}{2} \left( F_{\overline{k+1}}q^{-k_{k+1}}(f_{k+1}f_k-f_kf_{k+1})-q^{-k_{k+1}}(f_{k+1}f_k-f_kf_{k+1})F_{\overline{k+1}} \right)\\
&= -\frac{(q-q^{-1})^2}{2}q^{-k_{k+1}}[q^{-2}F_{\overline{k+1}}(f_{k+1}f_k-f_kf_{k+1})-(f_{k+1}f_k-f_kf_{k+1})F_{\overline{k+1}}]\\
&= -\frac{(q-q^{-1})^2}{2}q^{-k_{k+1}}[F_{\overline{k+1}}f_{k+1}f_k-F_{\overline{k+1}}f_kf_{k+1} -f_{k+1}f_kF_{\overline{k+1}}+f_kf_{k+1}F_{\overline{k+1}} \\ &+ (q^{-2}-1)F_{\overline{k+1}}(f_{k+1}f_k-f_kf_{k+1})]\\
&= -\frac{(q-q^{-1})^2}{2}q^{-k_{k+1}}[q^{-1}f_{k+1}F_{\overline{k+1}}f_k-F_{\overline{k+1}}f_kf_{k+1} -f_{k+1}f_kF_{\overline{k+1}}+qf_kF_{\overline{k+1}}f_{k+1}\\ &+ (q^{-2}-1)F_{\overline{k+1}}(f_{k+1}f_k-f_kf_{k+1})]\\
&= -\frac{(q-q^{-1})^2}{2}q^{-k_{k+1}}[f_{k+1}(q^{-1}F_{\overline{k+1}}f_k - f_kF_{\overline{k+1}})-q(q^{-1}F_{\overline{k+1}}f_k - f_kF_{\overline{k+1}})f_{k+1}\\ &+ (q^{-2}-1)F_{\overline{k+1}}(f_{k+1}f_k-f_kf_{k+1})]\\
&= -\frac{(q-q^{-1})^2}{2}q^{-k_{k+1}}[2f_{k+1}q^{k_{k+1}}f_{\overline{k}}-2qq^{k_{k+1}}f_{\overline{k}}f_{k+1} + (q^{-2}-1)F_{\overline{k+1}}(f_{k+1}f_k-f_kf_{k+1})]\\
&= -\frac{(q-q^{-1})^2}{2}q^{-k_{k+1}}[2qq^{k_{k+1}}(f_{k+1}f_{\overline{k}}-f_{\overline{k}}f_{k+1}) + (q^{-2}-1)F_{\overline{k+1}}(f_{k+1}f_k-f_kf_{k+1})]\\
&= -\frac{(q-q^{-1})^2}{2}q^{-k_{k+1}}[(1-q^{-2})F_{\overline{k+1}}(f_{k+1}f_k-f_kf_{k+1}) +  (q^{-2}-1)F_{\overline{k+1}}(f_{k+1}f_k-f_kf_{k+1})]\\
&= 0.
\end{align*}
\end{small}
The induction step for the second induction  ($i - k \geq 2$) is established as follows:
\begin{align*}
t_{i,-i}t_{k,i+1} - t_{k,i+1}t_{i,-i} 
&= -\frac{q-q^{-1}}{2}[F_{\overline{i}}q^{-k_{k+1}}(f_kt_{k+1,i+1}-t_{k+1,i+1}f_k)-t_{k,i+1}F_{\overline{i}}]\\   
&= -\frac{q-q^{-1}}{2}[q^{-k_{k+1}}F_{\overline{i}}(f_kt_{k+1,i+1}-t_{k+1,i+1}f_k)-t_{k,i+1}F_{\overline{i}}]\\    
&= -\frac{q-q^{-1}}{2}[q^{-k_{k+1}}(f_kt_{k+1,i+1}-t_{k+1,i+1}f_k)F_{\overline{i}}-t_{k,i+1}F_{\overline{i}}]\\   
&= -\frac{q-q^{-1}}{2}[t_{k,i+1}F_{\overline{i}}-t_{k,i+1}F_{\overline{i}}]\\    
&= 0.
\end{align*}

For the induction step of the first induction ($\ell - i \geq 2$) we proceed as follows:
\begin{align*}
t_{i,-i}t_{k\ell} - t_{k\ell}t_{i,-i} &= -\frac{q-q^{-1}}{2}[F_{\overline{i}}q^{-k_{\ell-1}}(f_{\ell-1}t_{k,\ell-1}-t_{k,\ell-1}f_{\ell-1})-q^{-k_{\ell-1}}(f_{\ell-1}t_{k,\ell-1}-t_{k,\ell-1}f_{\ell-1})F_{\overline{i}}]\\
&= -\frac{q-q^{-1}}{2}[q^{-k_{\ell-1}}F_{\overline{i}}(f_{\ell-1}t_{k,\ell-1}-t_{k,\ell-1}f_{\ell-1})-q^{-k_{\ell-1}}(f_{\ell-1}t_{k,\ell-1}-t_{k,\ell-1}f_{\ell-1})F_{\overline{i}}]\\
&= -\frac{q-q^{-1}}{2}[q^{-k_{\ell-1}}(f_{\ell-1}t_{k,\ell-1}-t_{k,\ell-1}f_{\ell-1})F_{\overline{i}}-q^{-k_{\ell-1}}(f_{\ell-1}t_{k,\ell-1}-t_{k,\ell-1}f_{\ell-1})F_{\overline{i}}]\\
&= 0
\end{align*}

\noindent{\it Case 25.} Suppose that $\vert i \vert < \vert j \vert = \vert k \vert = \vert l \vert$ in (\ref{exprel}). We prove that 
\begin{align*}
0 =(-1)^{(p(i)+p(j))(p(j)+p(k))}q^{\text{sgn}(j)}t_{ij}t_{kj} - t_{kj}t_{ij}
\end{align*}
for $j=\ell$, and
\begin{align*}
0=(-1)^{(p(i)+p(j))(p(j)+p(k))}q^{\text{sgn}(j)}t_{ij}t_{k,-j} - t_{k,-j}t_{ij} + (-1)^{p(i)}\delta_{j>0}(q-q^{-1})t_{i,-j}t_{kj}
\end{align*}
for $j=-\ell$, using some of the relations in the proposition. We proceed by  induction on $|j| - |i|$ and  consider only the case $i > 0$ and $j=k > 0$ as the other cases follow similarly. 

For the base case $j = i + 1$, the relations
\begin{align*}
qt_{i,i+1}t_{i+1,i+1} = t_{i+1,i+1}t_{i,i+1}
\end{align*}
when $j=\ell$, and 
\begin{align*}
qt_{i,i+1}t_{i+1,-i-1} - t_{i+1,-i-1}t_{i,i+1} = -(q-q^{-1})t_{i,-i-1}t_{i+1,i+1}
\end{align*}
when $j=-\ell$, follow from the relations $q^{k_{i+1}}f_{i} = qf_{i}q^{k_{i+1}}$ and $f_iF_{\overline{i+1}} - q^{-1}F_{\overline{i+1}}f_i=-2q^{k_{i+1}}f_{\overline{i}} $.  For the induction step ($j - i \geq 2$) we have:
\begin{align*}
qt_{ij}t_{jj} 
&= qq^{-k_{i+1}}(f_{i}t_{i+1,j}-t_{i+1,j}f_i)q^{k_j}\\
&= q^{k_j}q^{-k_{i+1}}(f_{i}t_{i+1,j}-t_{i+1,j}f_i)\\
&= t_{jj}t_{ij}
\end{align*}
for $j=\ell$, and
\begin{align*}
qt_{ij}t_{j,-j} &= -q\frac{q-q^{-1}}{2}q^{-k_{i+1}}(f_{i}t_{i+1,j}-t_{i+1,j}f_i)F_{\overline{j}}\\
&= -q\frac{q-q^{-1}}{2}q^{-k_{i+1}}f_{i}t_{i+1,j}F_{\overline{j}}+q\frac{q-q^{-1}}{2}q^{-k_{i+1}}t_{i+1,j}f_iF_{\overline{j}}\\
&= -\frac{q-q^{-1}}{2}q^{-k_{i+1}}f_{i}(F_{\overline{j}}t_{i+1,j}+2t_{i+1,-j}q^{k_j})+\frac{q-q^{-1}}{2}q^{-k_{i+1}}(F_{\overline{j}}t_{i+1,j}+2t_{i+1,-j}q^{k_j})f_i\\
&= \frac{q-q^{-1}}{2}q^{-k_{i+1}}F_{\overline{j}}(f_it_{i+1,j}-t_{i+1,j}f_i) -(q-q^{-1})q^{-k_{i+1}}(f_it_{i+1,-j}-t_{i+1,-j}f_i )q^{k_j}\\
&= \frac{q-q^{-1}}{2}F_{\overline{j}}q^{-k_{i+1}}(f_it_{i+1,j}-t_{i+1,j}f_i) -(q-q^{-1})q^{-k_{i+1}}(f_it_{i+1,-j}-t_{i+1,-j}f_i )q^{k_j}\\
&= t_{j,-j}t_{ij} -(q-q^{-1})t_{i,-j}t_{jj}
\end{align*}
for $j=-\ell$.    \end{proof}

We define a standard grading onto $\mathfrak{U}_q\mathfrak{p}_n$ , namely we let $\deg e_i = \alpha_i$, $\deg f_i = -\alpha_i$, $\deg q^h = 0$, $\deg e_{\overline{i}} = \gamma_i$, $\deg f_{\overline{i}} = -\gamma_i$, and $\deg F_{\overline{i}} = -\beta_i$. With this grading, all of the defining relations of the quantum superalgebra $\mfU_q\mfp_n$ are homogeneous. Hence, we say that $\mathfrak{U}_q\mathfrak{p}_n$ have a $Q$-grading 
\begin{align*}
\mfU_q\mfp_n = \bigoplus_{\alpha\in Q} (\mfU_q)_\alpha,
\end{align*}
where $(\mfU_q)_\alpha = \left \{ v\in \mfU_q\mfp_n \mid q^hvq^{-h} = q^{\alpha(h)}v \text{ for all } h\in P^\vee  \right\}$. In what follows we write $\deg u = \alpha$ whenever $u \in (\mfU_q)_\alpha$.

The comultiplication $\Delta$ of $\mfU_q\mfp_n$ is given in \cite{AGG} by the formula
\begin{align*}
\Delta(t_{ij})=\sum_{k=-n}^n (-1)^{(p(i) + p(k))(p(k)+p(j))} t_{ik} \ot t_{kj}.
\end{align*}

Through direct computations, we can express the comultiplication $\Delta$ in terms of the new generators in Proposition \ref{prop:alg}. The details are left to the reader.

\begin{lemma} \label{lem:comult} In terms of the generators $e_i, f_i, e_{\overline{i}}, f_{\overline{i}}$ for $i\in I$, $F_{\overline{i}}$ for $i\in I \cup \{ n\}$, and $q^h$ for $h\in P^{\vee}$, the following hold in $\mfU_q\mfp_n$:
\begin{align*}
\Delta(q^h) &= q^h \ot q^h,\\
\Delta(e_i) &=  q^{k_i}\ot e_i + e_i \ot q^{k_{i+1}}  - \frac{q-q^{-1}}{2}e_{\overline{i}}\ot F_{\overline{i+1}}, \\
\Delta(f_i) &=   q^{k_i} \ot f_i + f_i \ot q^{k_{i+1}} + \frac{q-q^{-1}}{2} F_{\overline{i}}  \ot e_{\overline{i+1}},\\
\Delta(e_{\overline{i}}) &=   q^{k_i}  \ot e_{\overline{i}} + e_{\overline{i}} \ot q^{k_{i+1}},  \\
\Delta(f_{\overline{i}}) &=  q^{k_i} \ot f_{\overline{i}} + f_{\overline{i}} \ot q^{k_{i+1}} - \frac{q-q^{-1}}{2} F_{\overline{i}} \ot e_i + \frac{q-q^{-1}}{2}f_i \ot F_{\overline{i+1}}, \\
\Delta(F_{\overline{i}}) &=   q^{k_i}  \ot F_{\overline{i}} + F_{\overline{i}} \ot q^{k_{i}}.  \\
\end{align*}
\end{lemma}

Let $\mfU_q^+$ (respectively $\mfU_q^-$) be the subsuperalgebra of $\mfU_q\mfp_n$ generated by the elements $e_i$ and $e_{\overline{i}}$ (respectively $f_i$, $f_{\overline{i}}$, and $F_{\overline{j}}$) for $i \in I$ (and $j\in I \cup \{ n\}$). Also, let $\mfU_q^0$ be the subsuperalgebra of $\mfU_q\mfp_n$ generated by $q^h$ for $h\in P^\vee$. In order to establish the  triangular decomposition of $\mfU_q\mfp_n$ we need the following lemma.

\begin{lemma} \label{lem:tridecomp}
Let $\mathfrak{U}_q^{\geq 0}$ (respectively, $\mathfrak{U}_q^{\leq 0}$) be generated by $\mathfrak{U}_q^0$ and $\mathfrak{U}_q^+ $ (respectively,  $\mathfrak{U}_q^0$ and $\mathfrak{U}_q^-$). Then the following $\C(q)$-linear isomorphisms hold.
\begin{align*}
\mathfrak{U}_q^{\geq 0} \cong \mathfrak{U}_q^0 \otimes \mathfrak{U}_q^+, \hspace{30pt} \mathfrak{U}_q^{\leq 0} \cong \mathfrak{U}_q^- \otimes \mathfrak{U}_q^0.
\end{align*}

\end{lemma}

\begin{proof}
We will prove the second isomorphism. Let $\{f_\zeta \mid \zeta \in \Omega\}$ be a basis of $\mfU_q^-$ consisting of monomials in $f_i$'s, $f_{\overline{i}}$'s, and $F_{\overline{j}}$'s ($1\leq i \leq n-1$, $1\leq j \leq n$), with $\Omega$ being an index set. Consider the map $\varphi: \mathfrak{U}_q^- \otimes \mathfrak{U}_q^0 \rightarrow \mathfrak{U}_q^{\leq 0} $ defined by $\varphi(f_\zeta \otimes q^h) = f_\zeta q^h$. The defining relations of $\mfU_q\mfp_n$ imply that $f_\zeta q^h$ span $\mathfrak{U}_q^{\leq 0} $, so $\varphi$ is surjective. It remains to show that the set $\{f_\zeta q^h \mid \zeta \in \Omega, h\in P^\vee\}$ is linearly independent over $\C(q)$.

Suppose 
\begin{align*}
\sum_{\substack{\zeta\in\Omega \\ h\in P^\vee}} C_{\zeta, h} f_\zeta q^h = 0,
\end{align*}
for some  $C_{\zeta, h} \in \C(q)$.  Then
\begin{align*}
\sum_{\alpha\in Q_-} \left(\sum_{\substack{\deg f_\zeta = \alpha \\ h\in P^\vee}} C_{\zeta, h} f_\zeta q^h\right) = 0.
\end{align*}
Write  $\displaystyle \alpha = -\sum_{i=1}^{n-1}(m_i\alpha_i+n_i\gamma_i) - \sum_{i=1}^{n}r_i\beta_i$, for $m_i,n_i, r_i\in \mathbb{Z}_{\geq 0}$, and let $\displaystyle h_\alpha = \sum_{i=1}^{n-1}(m_i+n_i)k_{i+1} + \sum_{i=1}^{n}r_ik_{i}$ and  $\displaystyle h_\alpha^\prime = \sum_{i=1}^{n-1}(m_i+n_i)k_{i} +r_ik_{i}$. From $\displaystyle \mfU_q\mfp_n = \bigoplus_{\alpha\in Q} (\mfU_q)_\alpha$, we have that, for each $\alpha\in Q_-$,
\begin{align}\label{sum2}
\sum_{\substack{\deg f_\zeta = \alpha \\ h\in P^\vee}} C_{\zeta, h} f_\zeta q^h = 0.
\end{align}
Since $f_\zeta$ is a monomial in $f_i$'s, $f_{\overline{i}}$'s, and $F_{\overline{i}}$'s, we have

\begin{align*}
\Delta(f_\zeta) = f_\zeta \otimes q^{h_{\alpha}} + \ldots + q^{h_{\alpha}^\prime} \otimes f_\zeta.
\end{align*}

Hence, the degree $(\alpha, 0)$ term in the decomposition of  $\Delta(f_\zeta)$ equals $ f_\zeta \otimes q^{h_{\alpha}}$. Applying the comultiplication to (\ref{sum2}) gives

\begin{align*}
\sum_{\substack{\deg f_\zeta = \alpha \\ h\in P^\vee}} C_{\zeta, h} ( f_\zeta q^{h} \otimes q^{h+h_{\alpha}} + \hdots + q^{h+h_{\alpha}^\prime} \otimes f_\zeta q^{h} ) = 0.
\end{align*}

Collecting the terms of degree $(\alpha, 0)$ gives that

\begin{align*}
\sum_{\substack{\deg f_\zeta = \alpha \\ h\in P^\vee}} C_{\zeta, h}  (f_\zeta q^{h} \otimes q^{h+h_{\alpha}})= 0.
\end{align*}

Since for every $\alpha$, the set $\{q^{h+h_{\alpha}} \mid h \in P^\vee\}$ is linearly independent, we have that, for all $h\in P^\vee$: 

\begin{align*}
& \sum_{\substack{\deg f_\zeta = \alpha }} C_{\zeta, h} f_\zeta q^{h} = 0.
\end{align*}

Due to the linear independence of $f_\zeta$, we conclude that $C_{\zeta, h} =0$ for all $\zeta, h$,  as desired.\end{proof}

\begin{theorem} \label{thm:tridecomp} There is a $\C(q)$-linear isomorphism 
\begin{align*}
\mfU_q(\mfp_n) \cong \mfU_q^- \otimes \mfU_q^0 \otimes \mfU_q^+.
\end{align*}
\end{theorem}
\begin{proof}
Let $\{f_\zeta \mid \zeta \in \Omega\}$, $\{q^h \mid h \in P^\vee\}$, and $\{e_\tau \mid \tau \in \Omega^\prime\}$ be monomial bases of $\mfU_q^-$, $\mfU_q^0$, and $\mfU_q^+$, respectively, where $\Omega$ is the index set as in the proof of Lemma \ref{lem:tridecomp}, and $\Omega^\prime$ is another  index set. Using the defining relations of $\mfU_q\mfp_n$ in Proposition \ref{prop:alg}, we can  express every monomial in $\mfU_q(\mfp_n)$ as a linear combination of monomials each of which has   $e_i$ and $e_{\overline{i}}$ on the right. By Lemma \ref{lem:tridecomp}, the monomials   $f_{\zeta}q^he_{\tau}$ span $\mfU_q\mfp_n$. Hence, it remains to show that  $f_{\zeta}q^he_{\tau}$ are linearly independent over $\C(q)$. 

Suppose
\begin{align*}
\sum_{\substack{\zeta\in\Omega, \tau\in \Omega^\prime \\ h\in P^\vee}} C_{\zeta, h,\tau} f_\zeta q^h e_{\tau} = 0,
\end{align*}
where $C_{\zeta, h,\tau}$ is some nonzero constant in $\C(q)$. Due to the $Q$-grading of $\mfU_q\mfp_n$, we have that, for all $\alpha \in Q$:
\begin{align} \label{sum1}
\sum_{\substack{\deg f_\zeta + \deg e_\tau = \alpha \\ h\in P^\vee}} C_{\zeta, h,\tau} f_\zeta q^h e_{\tau} = 0.
\end{align}

Define a partial ordering on $\mfh^*$ by $\lambda \leq \mu $ if and only if $\lambda - \mu \in Q_-$ for $\lambda,\mu \in \mfh^*$. We then choose $\gamma = \deg f_{\zeta}$ and $\beta = \deg e_{\tau}$, which are minimal and maximal, respectively, among those for which $\gamma + \beta = \alpha$ and $C_{\zeta, h, \tau} \neq 0$. If $\displaystyle \gamma = -\sum_{i=1}^{n-1}(m_i\alpha_{i}+n_i\gamma_{i}) - \sum_{i=1}^{n}r_i\beta_{i}$, set $\displaystyle h_\gamma = \sum_{i=1}^{n-1}(m_i+n_i)k_{i+1} + \sum_{i=1}^{n}r_ik_{i}$, and if $\displaystyle \beta = \sum_{i=1}^{n-1}(m^\prime_i\alpha_{i}+n^\prime_i\gamma_{i})$, set $\displaystyle h_\beta = \sum_{i=1}^{n-1}(m^\prime_ik_{i}+n^\prime_ik_{i})$, for $m_i,m^\prime_i,n_i,n^\prime_i, r_i\in \Z_{\geq 0}$.

The term of degree $(0,\beta)$ in $\Delta(e_\tau)$ is $q^{h_\beta}\otimes e_{\tau}$ and the term of degree $(\gamma,0)$ in $\Delta(f_\zeta)$ is $f_{\zeta} \otimes q^{h_\gamma}$. Applying the comultiplication to the sum in (\ref{sum1}), and looking at the terms of degree $(\gamma, \beta)$, we have that
\begin{align*}
\sum_{\substack{\deg f_\zeta = \gamma \\ \deg e_\tau = \beta \\ h\in P^\vee}} C_{\zeta, h,\tau}  (f_\zeta q^{h+h_\beta} \otimes q^{h+h_{\gamma}}e_\tau)= 0.
\end{align*}

By Lemma \ref{lem:tridecomp}, the elements $f_\zeta q^h$ are linearly independent for $\zeta\in\Omega, h\in P^\vee$. Thus, for all $h\in P^\vee$, we have that

\begin{align*}
& \sum_{\substack{\deg e_\tau = \beta }} C_{\zeta, h, \tau} q^{h+h_\gamma}e_\tau = 0.
\end{align*}

Due to the linear independence of $e_\tau$, we conclude that  $C_{\zeta, h, \tau} =0$, leading to contradiction. Therefore all coefficients in \eqref{sum1} are zero.
\end{proof}

\section{Highest weight representation theory of $\mfU_q\mfp_n$}\label{sec:highest-weight-uqpn}

A $\mfU_q\mfp_n$-module $V^q$ is called a \textit{weight module} if it admits a weight space decomposition
\begin{align*}
V^q = \bigoplus\limits_{\mu\in P} V_\mu^q
\end{align*}
where $V_\mu^q = \{v\in V^q \mid q^hv = q^{\mu(h)}v \text{ for all } h\in P^\vee\}$ is the $\mu$-\textit{weight space}. We call $\mu \in P$  a \textit{weight} of $V^q$ if $V_\mu^q \neq 0$. A nonzero vector $v\in V_\mu^q$ is called a \textit{weight vector} of weight $\mu$. If $v\in V^q$ is a nonzero vector such that $\mfU_q^+v = 0$, then $v$ is called a \textit{maximal vector}. 

The dimension of each weight space $\dim V_\mu^q$ is called the \textit{weight multiplicity} of $\mu$. If $\dim V_\mu^q < \infty$ for all $\mu \in P$, the \textit{character} of $V^q$ is defined by 
\begin{align*}
\ch V^q = \sum\limits_\mu \dim_\mbbF V_\mu^q e^\mu
\end{align*}
where $\mbbF = \C (q)$ and $e^\mu$ are formal basis elements of the group algebra $\mbbF[P]$ with multiplication defined by $e^\alpha e^\beta = e^{\alpha+\beta}$.

The above definitions and notions can be introduced in the same way for $\mfU\mfp_n$-modules
over $\mbbF = \C$.

\begin{definition} A weight $\mfU_q\mfp_n$-module $V^q$ is called a \emph{highest weight module} with highest weight $\lambda \in P$ if the following holds for some nonzero $v\in V^q$:
\begin{itemize}
\item[(a)] $v$ is a maximal vector of $V^q$,
\item[(b)] $v\in V_\lambda^q$, and
\item[(c)] $V^q = \mfU_q\mfp_n v$.
\end{itemize}

This vector $v$, which is unique up to a constant multiple, is called a \emph{highest weight vector} of $V^q$.
\end{definition}

This definition, along with Theorem \ref{thm:tridecomp}, shows that $V^q = \mfU_q^- v$ for any highest weight module with highest weight vector $v$ and highest weight $\lambda$.

Fix $\lambda \in P$ and define $J^q(\lambda)$ to be the left ideal of $\mfU_q\mfp_n$ generated by $e_i$, $e_{\overline{i}}$, and $q^h - q^{\lambda(h)}1$, for $i\in I$ and $h\in P^\vee$. Then  $M^q(\lambda) = \mfU_q\mfp_n / J^q(\lambda)$ is the Verma module, which is a $\mfU_q\mfp_n$-module by left multiplication. Set $v = 1 + J^q(\lambda)$. Then $M^q(\lambda)$ is a highest weight module with highest weight $\lambda$ and highest weight vector $v$. The proof of the following proposition is standard. See, for example, the proof of Proposition 3.2.2 in \cite{HK}, which uses the same arguments.
\begin{proposition} \label{prop-verma}  \hspace{2pt} \\ \vspace{-20pt}
\begin{itemize}
\item[(a)] $M^q(\lambda)$ is a free $\mfU_q^-$-module of rank 1, generated by the highest weight vector $v = 1 + J^q(\lambda)$.
\item[(b)] Every highest weight $\mfU_q\mfp_n$-module with highest weight $\lambda$ is a homomorphic image of $M^q(\lambda)$.
\item[(c)] The Verma module $M^q(\lambda)$ has a unique maximal submodule.
\end{itemize}

\end{proposition}

Let $N_q(\lambda)$ denote the unique maximal submodule of the Verma module $M^q(\lambda)$ from Proposition \ref{prop-verma}(c). Then the unique irreducible quotient 
\begin{align*}
V^q(\lambda) = M^q(\lambda) / N_q(\lambda)
\end{align*}
is the irreducible highest weight module over $\mfU_q\mfp_n$ with highest weight $\lambda$.

We note again that the definitions of highest weight module can be introduced in the same way for $\mfU\mfp_n$-modules over $\mathbb F = \C$. In the latter case we will use the notation $M(\lambda)$ and $V(\lambda)$ for the Verma module and its irreducible quotient, respectively.  We denote by  $\Lambda^+$ the set of $\mfp_n$-dominant integral weights:
\begin{align*}
\Lambda^+ &\coloneqq \{\lambda_1\epsilon_1 + \hdots + \lambda_n\epsilon_n \in \mfh^* \mid \lambda_i - \lambda_{i+1} \in \Z_{\geq 0}, \, \forall i \in I\}.
\end{align*}
The following proposition and theorem will be used to prove an important result concerning highest weight modules over $\mfU_q\mfp_n$.

\begin{proposition} \label{prop:irrhighweightfinitedimclassical}  Let $V$ be a highest weight $\mathfrak{p}_n$-module with highest weight $\lambda\in \Lambda^+$ and highest weight vector $v$ such that $f_i^{\lambda(k_i)-\lambda(k_{i+1}) + 1}v = 0$ for all $i\in I$. Then $V$ is finite dimensional.

\end{proposition}
\begin{proof}
The proof is of this is similar to that of Proposition 1.9 in \cite{GJKK}. The main idea of the proof is that since $e^2_{\overline{i}} = 0$ and $f^2_{\overline{i}} = 0$, then, with the aid of the Poincar\'e-Birkhoff-Witt theorem, we show that $\mfU_q(\mfp_n)_0v$, where $(\mfp_n)_0$ is generated by $\{e_i, f_i, k_j \mid i\in I, j\in J\}$, is finite generated, using the fact that $f_i^{\lambda(k_i)-\lambda(k_{i+1}) + 1}v = 0$ and $v$ is a highest weight vector. For details, see the proof of Proposition 1.9 in \cite{GJKK}.
\end{proof}

\begin{theorem}[\cite{K2}] \label{thm:dominant-weight-fin-dim-mod} For any weight $\lambda\in \mfh^*$, $V(\lambda)$ is finite dimensional if and only if $\lambda \in \Lambda^+$. \end{theorem}

We recall some standard definitions from $q$-calculus.  We set
\begin{align*}
[n]_q \coloneqq \frac{q^n - q^{-n}}{q-q^{-1}}.
\end{align*}
We also define $[0]_q! \coloneqq 1$, and $[n]_q! = [n]_q \cdot [n-1]_q \cdot \hdots \cdot [1]_q$. The \textit{divided powers} of $e_i$ and $f_i$ are:
\begin{align*}
e_{i}^{(m)} \coloneqq \frac{e^m_i}{[m]_q!}, \hspace{30pt} f_{i}^{(m)} \coloneqq \frac{f^m_i}{[m]_q!}
\end{align*}

\begin{lemma} \label{lem:action} For all $i\in I$ and $m\in \Z_{\geq 0}$, we have
\begin{itemize}
\item[(a)] \begin{align*}
e_if_i^{(m)} = f_i^{(m)}e_i - f_i^{(m-1)}\frac{q^{-m+1}q^{2k_i}-q^{m-1}q^{2k_{i+1}}}{q^2-1} + (1-q^{-2})\left(q^{m-1}f_{i}^{(m-1)} f_{\overline{i}}+ \frac{q^{2m-2}}{2}q^{k_i}F_{\overline{i}}f_i^{(m-2)}\right)e_{\overline{i}},
\end{align*}
\item[(b)] \begin{align*}
f_ie_i^{(m)} = e_i^{(m)}f_i + q^2e_i^{(m-1)}\frac{q^{m-1}q^{2k_i}-q^{-m+1}q^{2k_{i+1}}}{q^2-1} + (q^2-1)\left(q^2e_i^{(m-1)}e_{\overline{i}}f_{\overline{i}} - \frac{1}{2}e_i^{(m-2)}q^{k_{i+1}}e_{\overline{i}}F_{\overline{i+1}} \right)
\end{align*}
(in the case $m=1$ we assume that the terms involving $f_i^{(m-2)}$ and  $e_i^{(m-2)}$ are zero).
\end{itemize}

\begin{proof} We prove (a) by induction on $m$. The base case $m=1$:
\begin{align*}
e_if_i = f_ie_i - \frac{q^{2k_i}-q^{2k_{i+1}}}{q^2-1} + \frac{q^2-1}{q^2}f_{\overline{i}}e_{\overline{i}}
\end{align*}
follows from  Proposition \ref{prop:alg}. For the base case $m=2$ we multiply the above relation on the right by $f_i$ and obtain:
\begin{align*}
 e_if_i^2 &= f_ie_if_i - \frac{q^{2k_i}-q^{2k_{i+1}}}{q^2-1}f_i + \frac{q^2-1}{q^2}f_{\overline{i}}e_{\overline{i}}f_i\\
 &= f_i^2e_i - f_i\frac{q^{2k_i}-q^{2k_{i+1}}}{q^2-1} + (1-q^{-2})f_if_{\overline{i}}e_{\overline{i}} - f_i\frac{q^{-2}q^{2k_i}-q^2q^{2k_{i+1}}}{q^2-1} + (1-q^{-2})f_{\overline{i}}e_{\overline{i}}f_i\\
 &= f_i^2e_i - f_i\frac{(1+q^{-2})q^{2k_i}-(1+q^2)q^{2k_{i+1}}}{q^2-1} + (1-q^{-2})\left(f_if_{\overline{i}} + f_{\overline{i}}f_i\right)e_{\overline{i}}\\
&= f_i^2e_i - f_i\frac{(1+q^{-2})q^{2k_i}-(1+q^2)q^{2k_{i+1}}}{q^2-1} + (1-q^{-2})\left((1+q^{2})f_if_{\overline{i}}+ \frac{1+q^{2}}{2}qq^{k_i}F_{\overline{i}}\right)e_{\overline{i}}.
\end{align*}

Dividing both sides by $[2]_q$ leads to the desired result for $m=2$.

Now, suppose that (a) is true for some $m$. Then we have that
{\allowdisplaybreaks
\begin{align*}
 e_if_i^{(m)}f_i &=  f_i^{(m)}e_if_i - f_i^{(m-1)}\frac{q^{-m+1}q^{2k_i}-q^{m-1}q^{2k_{i+1}}}{q^2-1}f_i + (1-q^{-2})\left(q^{m-1}f_{i}^{(m-1)}f_{\overline{i}} + \frac{q^{2m-2}}{2}q^{k_i}F_{\overline{i}}f_i^{(m-2)}\right)e_{\overline{i}}f_i\\
&=  f_i^{(m)}  f_ie_i - f_i^{(m)} \frac{q^{2k_i}-q^{2k_{i+1}}}{q^2-1} + \frac{q^2-1}{q^2}f_i^{(m)} f_{\overline{i}}e_{\overline{i}} - f_i^{(m-1)}\frac{q^{-m+1}q^{2k_i}-q^{m-1}q^{2k_{i+1}}}{q^2-1}f_i \\ &+ (1-q^{-2})\left(q^{m-1}f_{i}^{(m-1)}f_{\overline{i}}f_i + \frac{q^{2m-2}}{2}q^{k_i}F_{\overline{i}}f_i^{(m-2)}f_i\right)e_{\overline{i}}\\
&=  f_i^{(m)}  f_ie_i - f_i^{(m)} \frac{q^{2k_i}-q^{2k_{i+1}}}{q^2-1} - f_i^{(m-1)}f_i\frac{q^{-m-1}q^{2k_i}-q^{m+1}q^{2k_{i+1}}}{q^2-1} \\ &+ (1-q^{-2})\left(f_i^{(m)} f_{\overline{i}} + q^{m-1}f_{i}^{(m-1)}f_{\overline{i}}f_i + \frac{q^{2m-2}}{2}q^{k_i}F_{\overline{i}}f_i^{(m-2)}f_i\right)e_{\overline{i}}\\
&=  f_i^{(m)}  f_ie_i - f_i^{(m)} \frac{q^{2k_i}-q^{2k_{i+1}}}{q^2-1} - f_i^{(m-1)}f_i\frac{q^{-m-1}q^{2k_i}-q^{m+1}q^{2k_{i+1}}}{q^2-1} \\ &+ (1-q^{-2})\left(f_i^{(m)} f_{\overline{i}} + q^{m+1}f_{i}^{(m-1)}f_if_{\overline{i}} + q^{m}\frac{1+q^2}{2}f_{i}^{(m-1)}q^{k_i}F_{\overline{i}} + \frac{q^{2m-2}}{2}q^{k_i}F_{\overline{i}}f_i^{(m-2)}f_i\right)e_{\overline{i}}\\
&= [m+1]_qf_i^{(m+1)} e_i - f_i^{(m)} \frac{(1+q^{-m-1}[m]_q)q^{2k_i}-(1+q^{m+1}[m]_q)q^{2k_{i+1}}}{q^2-1} \\ &+ (1-q^{-2})\left((1+ q^{m-1}[m]_q)f_{i}^{(m)}f_{\overline{i}} + q^{3m-2}\frac{1+q^2}{2}q^{k_i}F_{\overline{i}} f_{i}^{(m-1)}+ \frac{q^{2m-2}[m-1]_q}{2}q^{k_i}F_{\overline{i}}f_i^{(m-1)}\right)e_{\overline{i}}.
\end{align*}
}
Dividing both sides by $[m+1]_q$ completes the proof of (a).

The proof of (b) follows in a similar way, with the only main difference that we use the relation
from Lemma \ref{lem:alg}(c).
\end{proof}
\end{lemma}

\begin{proposition} \label{prop:irrhighweightfinitedim} Let $\lambda\in \Lambda^+$ and $V^q(\lambda)$ be  generated by a highest weight vector $v$. Then $f_i^{\lambda(k_i) - \lambda(k_{i+1})+1} v = 0$ for all $i\in I$.

\end{proposition}

\begin{proof} 
If $f_i^{\lambda(k_i) - \lambda(k_{i+1})+1} v \neq 0$ and $\mfU_q^+f_i^{\lambda(k_i) - \lambda(k_{i+1})+1} v = 0$, then $f_i^{\lambda(k_i) - \lambda(k_{i+1})+1} v $ generates a nontrivial proper submodule of $V^q(\lambda)$, which would be a contradiction to the fact that $V^q(\lambda)$ is irreducible. Therefore, it is enough to show that $\mfU_q^+f_i^{\lambda(k_i) - \lambda(k_{i+1})+1} v = 0$.

Note that for $j\neq i-1,i$, we have that
\begin{align*}
e_jf_i^{\lambda(k_i) - \lambda(k_{i+1})+1}v = f_i^{\lambda(k_i) - \lambda(k_{i+1})+1}e_j v = 0.
\end{align*} 

We next prove that  $e_jf_i^{\lambda(k_i) - \lambda(k_{i+1})+1}v = 0$ for $j=i,i-1$. In the case $j=i$, we use Lemma \ref{lem:action}(a) for $m= \lambda(k_i) - \lambda(k_{i+1})+1$:\begin{align*}
e_if_i^{(m)}v = \frac{q^{m-1+2\lambda(k_{i+1})}-q^{-m+1+2\lambda(k_i)}}{q^2-1} f_i^{(m-1)}v =0.
\end{align*}

For $j=i-1$ we apply induction on $m$. The base case $m=1$ follows from 
\begin{align*}
e_{i-1}f_iv =f_ie_{i-1}v - \frac{1-q^{-2}}{2q}q^{-k_{i}}F_{\overline{i}}(e_{\overline{i}}e_{i-1} - e_{i-1}e_{\overline{i}}) v = 0.
\end{align*}
Assume that  $e_{i-1}f_i^{ m} v =0$. Then we have 
\begin{align*}
e_{i-1}f_i^{ m+1} v &= f_ie_{i-1}f_i^{ m} v - \frac{1-q^{-2}}{2q}q^{-k_{i}}F_{\overline{i}}(e_{\overline{i}}e_{i-1} - e_{i-1}e_{\overline{i}})f_i^{ m} v \\
&=  f_ie_{i-1}f_i^{ m} v - \frac{1-q^{-2}}{2q}q^{-k_{i}}F_{\overline{i}}e_{\overline{i}}e_{i-1}f_i^{ m} v  + \frac{1-q^{-2}}{2q}q^{-k_{i}}F_{\overline{i}}e_{i-1}f_i^{ m} e_{\overline{i}}v \\
&=  \left(f_i- \frac{1-q^{-2}}{2q}q^{-k_{i}}F_{\overline{i}}e_{\overline{i}}\right)e_{i-1}f_i^{ m} v =0.
\end{align*}
Therefore, $e_jf_i^{ \lambda(k_i) - \lambda(k_{i+1})+1}v  = 0$ for all $j\in I$.

Lastly, we prove that  $e_{\overline{j}}f_i^{\lambda(k_i) - \lambda(k_{i+1})+1}v = 0$ for all $j \in I$. Note that for $j\neq i-1$,
\begin{align*}
e_{\overline{j}}f_i^{ \lambda(k_i) - \lambda(k_{i+1})+1}v = f_i^{ \lambda(k_i) - \lambda(k_{i+1})+1}e_{\overline{j}}v = 0.
\end{align*}

For $j = i-1$, using that $e_{i-1}f_i^{ m} v = 0$, we have
\begin{align*}
e_{\overline{i-1}}f_{i}^{m+1}v = f_ie_{\overline{i-1}}f_i^mv - e_{\overline{i}}e_{i-1}f_i^mv + e_{i-1}e_{\overline{i}}f_i^mv =  f_ie_{\overline{i-1}}f_i^mv.
\end{align*}
Thus we obtain $e_{\overline{i-1}}f_i^mv  = 0$ by induction on $m$. Hence, $\mfU_q^+f_i^{\lambda(k_i) - \lambda(k_{i+1})+1} v = 0$ and the proof is complete.\end{proof}

\section{Classical Limits}\label{sec:classical-limit}

Let $\mbA_1$ be the localization of $\C[q]$ at the ideal generated by $q-1$. Namely,
\begin{align*}
\mbA_1 = \{f(q)\in \C(q) \mid f \text{ is regular at } q=1\}.
\end{align*}

For an integer $n\in \Z$, we define
\begin{align*}
[x;n]_y \coloneqq \frac{xy^n - x^{-1}y^{-n}}{y-y^{-1}}, \hspace{30pt}   (x;n)_y \coloneqq \frac{xy^n - 1}{y-1}.
\end{align*}
In particular, $[q^h;0]_q , (q^h;0)_q \in \mfU_q^0$.

\begin{definition} The $\mbA_1$-form of $\mfU_q\mfp_n$, denoted by $\mfU_{\mbA_1}$,  is the $\mbA_1$-subalgebra of $\mfU_q\mfp_n$  generated by  the elements $e_i, f_i, e_{\overline{i}}, f_{\overline{i}}$ for $i\in I$, $F_{\overline{j}}$ for $j\in I \cup \{ n\}$, and $q^h$, $(q^h;0)_q$ for $h\in P^{\vee}$.
\end{definition}

Let $\mfU_{\mbA_1}^+$ (respectively, $\mfU_{\mbA_1}^-$), be the $\mbA_1$-subalgebra of $\mfU_{\mbA_1}$ generated by the elements $e_i$ and $e_{\overline{i}}$ for $i\in I$ (respectively, $f_i$, $f_{\overline{i}}$, and $F_{\overline{j}}$ for $i\in I$ and $j\in I \cup \{ n\}$). Let $\mfU_{\mbA_1}^0$ be the $\mbA_1$-subalgebra of $\mfU_{\mbA_1}$ generated by $q^h$ and $(q^h;0)_q$ for $h\in P^\vee$.

We will show that the triangular decomposition of $\mfU_q\mfp_n$ carries over to its $\mbA_1$-form. For this we first use the following lemma, whose proof is identical to the one of Lemma 5.2 in \cite{GJKK}.

\begin{lemma} \label{lem:A1tridecomp} \hspace{2pt} \\ \vspace{-15pt}
\begin{itemize}
\item[(a)] $(q^h;n)_q \in \mfU_{\mbA_1}^0$ for all $n \in \Z$ and $h\in P^\vee$.
\item[(b)] $[q^h;0]_q \in \mfU_{\mbA_1}^0$ for all $h\in P^\vee$.
\end{itemize}
\end{lemma}

\begin{proposition} \label{prop:A1tridecomp} The triangular decomposition of $\mfU_q\mfp_n$ in Theorem \ref{thm:tridecomp} induces  an isomorphism of $\mbA_1$-modules
\begin{align*}
\mfU_{\mbA_1} \cong \mfU_{\mbA_1}^-\otimes \mfU_{\mbA_1}^0 \otimes \mfU_{\mbA_1}^+.
\end{align*}
\end{proposition}

\begin{proof} Consider the  isomorphism $\varphi : \mfU_q\mfp_n \longrightarrow \mfU_q^-\otimes \mfU_q^0\otimes \mfU_q^+$ from Theorem \ref{thm:tridecomp}. Note that the following relations hold:
\begin{align*}
&e_i(q_h;0)_q = (q^h;-\alpha_i(h))_qe_i && e_{\overline{i}}(q_h;0)_q = (q^h;-\gamma_i(h))_qe_{\overline{i}} \\
&(q_h;0)_qf_i = f_i(q^h;-\alpha_i(h))_q && (q_h;0)_qf_{\overline{i}}  = f_{\overline{i}} (q^h;-\gamma_i(h))_q\\
&(q_h;0)_qF_{\overline{i}} = F_i(q^h;-\beta(h))_q \\
& e_if_i = f_ie_i - q^{-1}q^{k_i + k_{i+1}}\left[ q^{k_i-k_{i+1}} ; 0\right]_q + \frac{q^2-1}{q^2}f_{\overline{i}}e_{\overline{i}}\\
&e_if_{i+1} = f_{i+1}e_i - \frac{q-q^{-1}}{2}F_{\overline{i+1}}q^{-k_{i+1}}(e_{\overline{i+1}}e_i-e_ie_{\overline{i+1}}) && e_if_j = f_je_i \text{ for } j \neq i, i+1\\
&e_if_{\overline{i}} = q^{-1}f_{\overline{i}}e_i + \frac{q^2+1}{2q^3}F_{\overline{i+1}}q^{k_{i+1}} && e_{i+1}f_{\overline{i}} = f_{\overline{i}}e_{i+1} + f_{\overline{i+1}}f_i - f_if_{\overline{i+1}} \\
&e_if_{\overline{i+1}} = f_{\overline{i+1}}e_i + \frac{q-q^{-1}}{2}F_{\overline{i+1}}q^{-k_{i+1}}(e_{{i+1}}e_i-e_ie_{{i+1}}) && e_if_{\overline{j}} = f_{\overline{j}}e_i \text{ for } |i-j| > 1 \\
&e_{\overline{i}}f_{i+1} = f_{i+1}e_{\overline{i}}-e_{\overline{i+1}}e_i + e_ie_{\overline{i+1}} && e_{\overline{i}}f_j = f_je_{\overline{i}} \text{ for } j\neq i+1\\
& e_{\overline{i}}f_{\overline{i}} = -q^{-2}f_{\overline{i}}e_{\overline{i}} - q^{-1}q^{k_i + k_{i+1}}\left[ q^{k_i-k_{i+1}} ; 0\right]_q && e_{\overline{i}}f_{\overline{i+1}} = -f_{\overline{i+1}}e_{\overline{i}} + e_{i+1}e_i - e_ie_{i+1}\\
&e_{\overline{i+1}}f_{\overline{i}} = -f_{\overline{i}}e_{\overline{i+1}} + f_{i+1}f_i - f_if_{i+1} && e_{\overline{i}}f_{\overline{j}} =f_{\overline{j}}e_{\overline{i}} \text{ for } |i-j| > 1 \\
&e_iF_{\overline{i}} = q^{-1}F_{\overline{i}}e_i + 2q^{-1}f_{\overline{i}}q^{k_i} &&e_iF_{\overline{i+1}} = qF_{\overline{i+1}}e_i \\
&e_{\overline{i}}F_{\overline{i}} = -q^{-1}F_{\overline{i}}e_{\overline{i}} + 2q^{-1}f_{{i}}q^{k_i} && e_{\overline{i}}F_{\overline{i+1}} = -q^{-1}F_{\overline{i+1}}e_{\overline{i}} + 2q^{k_{i+1}}e_{{i}} \\
&e_iF_{\overline{j}} = F_{\overline{j}}e_i \text{ for } j \neq i, i+1 && e_{\overline{i}}F_{\overline{j}} = F_{\overline{j}}e_{\overline{i}} \text{ for } j \neq i, i+1
\end{align*}

These relations, together with Lemma \ref{lem:A1tridecomp}, imply that the image of the restriction $\overline{\varphi}$
of $\varphi$ to $\mfU_{\mbA_1}$ is a subset of $\mfU_{\mbA_1}^-\otimes \mfU_{\mbA_1}^0 \otimes \mfU_{\mbA_1}^+$. To define an inverse of $\overline{\varphi}$ we multiply the corresponding terms in the tensor product. This completes the proof. \end{proof}

\begin{definition}
The $\mbA_1$-form of the highest weight module $V^q$ with highest weight $\lambda \in P$ and highest weight vector $v$ is the $\mfU_{\mbA_1}$-module $V_{\mbA_1} =  \mfU_{\mbA_1}v$.
\end{definition}

For the rest of the section, by  $V^q$  we denote a highest weight module over $\mfU_q\mfp_n$ with highest weight $\lambda\in P$ and highest weight vector $v$. We can strengthen the above definition with the following proposition:

\begin{proposition} \label{prop:highestweightV} With the notation as above:
\begin{align*}
V_{\mbA_1} =  \mfU_{\mbA_1}^- v.
\end{align*}
\end{proposition}
 
\begin{proof}
In the light of Proposition \ref{prop:A1tridecomp}, it suffices to show that $ \mfU_{\mbA_1}^+ v = 0$ and $ \mfU_{\mbA_1}^0 v = \mbA_1v$. The former identity follows from the fact that $v$ is a highest weight vector. For the latter identity we use following:
\begin{align*}
q^hv = q^{\lambda(h)}v, \hspace{20pt} (q^h;0)_qv = \frac{q^{\lambda(h)}-1}{q-1}v.
\end{align*} \end{proof}

For each $\mu\in P$, we set $(V_{\mbA_1})_\mu = V_{\mbA_1}\cap V_\mu^q$. The following shows that the weight space decomposition of $V^q$ carries over to $V_{\mbA_1}$.
\begin{proposition}\label{prop:a1wsdecomp} $V_{\mbA_1}$ has the weight space decomposition $V_{\mbA_1} = \bigoplus\limits_{\mu\leq \lambda} (V_{\mbA_1})_\mu$. \end{proposition}
\begin{proof} The idea is standard but for reader's convenience we present the proof. Let $v = v_1+ \hdots + v_p \in V_{\mbA_1}$, where $v_j\in V_{\mu_j}^q$ and $\mu_j\in P$ for each $j\in \{1,2,\hdots, p\}$. We can assume that $\mu_j$ are distinct due to the weight decomposition of $V^q$. Therefore, it is enough to show that $v_j\in V_{\mbA_1}$ for each $j$.

Fix an index $j$. For each $i\neq j$, choose $h_i\in P^\vee$ such that $\mu_j(h_i) \neq \mu_i(h_i)$. Let $u\in \mfU_{\mbA_1}$ be defined by
\begin{align*}
u \coloneqq \prod_{i\neq j} \frac{\big( q^{h_i}; -\mu_i(h_i) \big)_q}{\big( q^{\mu_j(h_i)}; -\mu_i(h_i) \big)_q}.
\end{align*}
Then for each $i \neq j$:
\begin{align*}
\frac{\big( q^{h_i}; -\mu_i(h_i) \big)_q}{\big( q^{\mu_j(h_i)}; -\mu_i(h_i) \big)_q}v_k &= \frac{q^{-\mu_i(h_i)}q^{h_i} - 1}{q^{-\mu_i(h_i) }q^{\mu_j(h_i)} - 1}v_k\ = \frac{q^{-\mu_i(h_i)}q^{\mu_k(h_i)} - 1}{q^{-\mu_i(h_i) }q^{\mu_j(h_i)} - 1}v_k,
\end{align*}
where $k\in \{ 1, 2,\hdots, p\}$. Therefore $uv_j = v_j$ and $uv_i = 0$. Hence, $v_j = uv\in V_{\mbA_1}$.
\end{proof}

\begin{proposition} \label{prop:a1wsrank}
For each $\mu\in P$, the weight space $(V_{\mbA_1})_\mu$ is a free $\mbA_1$-module with $\rank_{\mbA_1}(V_{\mbA_1})_\mu = \dim_{\C(q)}V^q_\mu$.
\end{proposition}

\begin{proof} Since $\mbA_1$ is a principal ideal domain, every finitely generated torsion-free module over $\mbA_1$ is free. Notice that, for each $\mu\in P$, the weight space $(V_{\mbA_1})_\mu$ is finitely generated as an $\mbA_1$-module. The weight space is also torsion free, as otherwise it would contradict the fact that $\mbA_1$ is an integral domain. Thus, for each $\mu\in P$, the weight space $(V_{\mbA_1})_\mu$ is a free $\mbA_1$-module.

Since $\C(q)$ is the field of quotients of $\mbA_1$, a set of vectors of a $\C(q)$-vector space is linearly independent  if and only if it is $\mbA_1$-linearly independent. Thus $\rank_{\mbA_1}(V_{\mbA_1})_\mu \leq \dim_{\C(q)}V^q_\mu$. Let $f_\zeta$'s be some monomials in $f_i$'s, $f_{\overline{j}}$'s, and $F_{\overline{\ell}}$'s. Since $\{f_\zeta\}$ forms a linearly independent set over $\C(q)$, $\{f_{\zeta}v \; | \; \zeta \in Z\}$ is a $\C(q)$-basis of $V_\mu^q$ for an appropriate set $Z$. This basis is also contained in $(V_{\mbA_1})_\mu$ by definition. So we have that $\rank_{\mbA_1}(V_{\mbA_1})_\mu \geq \dim_{\C(q)}V^q_\mu$, which completes the proof. \end{proof}

Combining the previous two propositions gives us the following:

\begin{corollary} The map $\phi: \C(q) \otimes_{\mbA_1} V_{\mbA_1} \rightarrow V^q$, $f\otimes v \mapsto fv$, is an isomorphism of $\C(q)$-vector spaces.\end{corollary}

Let $\mbJ_1$ be the maximal ideal of $\mbA_1$ generated by $q-1$. Then there is a canonical isomorphism of fields 
\begin{align*}
\mbA_1 / \mbJ_1 \xlongrightarrow{\sim} \C \hspace{5pt} \text{ given by } f(q) + \mbJ_1 \longmapsto f(1).
\end{align*}

Define the $\C$-vector spaces
\begin{align*}
U_1 &= (\mbA_1/\mbJ_1) \otimes_{\mbA_1} \mfU_{\mbA_1},\\
V^1 &= (\mbA_1/\mbJ_1) \otimes_{\mbA_1} V_{\mbA_1}.
\end{align*}

Note that since $V_{\mbA_1}$ is a $\mfU_{\mbA_1}$-module, $V^1$ is naturally a $U_1$-module. Note that
\begin{align*}
U_1 \cong \mfU_{\mbA_1}/ \mbJ_1\mfU_{\mbA_1} \hspace{10pt} \text{ and } \hspace{10pt} V^1 \cong  V_{\mbA_1} / \mbJ_1 V_{\mbA_1},
\end{align*}

which gives rise to the following natural maps
\begin{align*}
\mfU_{\mbA_1} &\longrightarrow \mfU_{\mbA_1}/ \mbJ_1\mfU_{\mbA_1} \cong U_1, \\
V_{\mbA_1} &\longrightarrow V_{\mbA_1} / \mbJ_1 V_{\mbA_1}\cong V^1.
\end{align*}

Note that $q$ is mapped to $1$ under these maps, hence $U_1$ can be considered as the limit of $\mfU_q\mfp_n$ at $q=1$. The passage under these maps is referred to as taking the \textit{classical limit}. We write $\overline{x}$ for the image of $x$ under these maps.

Let $\overline{h}\in U_1$ denote the classical limit of the element $(q^h;0)_q\in \mfU_{\mbA_1}$.  The following is standard (see for example Lemma 3.4.3 in \cite{HK}):

\begin{lemma} \label{lem:classicalU0} \hspace{2pt} \\ \vspace{-15pt}
\begin{itemize}
\item[(a)] For all $h\in P^\vee$, $\overline{q^h} = 1$
\item[(b)] For any $h,h^\prime\in P^\vee$, $\overline{h+h^\prime} = \overline{h} + \overline{h^\prime}$.
\end{itemize}

\end{lemma}

This lemma shows that the image of $\mfU_{\mbA_1}^0$ under the classical limit is quite close to $U^0 = \mfU(\mfh)$.

For each $\mu\in P$, define $V_\mu^1 = (\mbA_1/\mbJ_1) \otimes_{\mbA_1} (V_{\mbA_1})_\mu \cong (V_{\mbA_1})_\mu / \mbJ_1 (V_{\mbA_1})_\mu $. 

\begin{lemma} \label{lem:v1wsdecomp} \hspace{2pt} \\ \vspace{-15pt}
\begin{itemize}
\item[(a)] For each $\mu\in P$, if $\{v_i \: | \; i=1,...,m\}$ is a basis of the free $\mbA_1$-module $(V_{\mbA_1})_\mu$, then $\{\overline{v}_i \: | \; i=1,...,m\}$ is a basis of the $\C$-vector space $V^1_\mu$.
\item[(b)] For each $\mu\in P$, a subset $\{v_i \: | \; i=1,...,m\}$ of $ (V_{\mbA_1})_\mu$ is $\mbA_1$-linearly independent if the  $\{\overline{v}_i \: | \; i=1,...,m\} \subset V^1_\mu$ is $\C$-linearly independent.
\end{itemize}

\end{lemma}

\begin{proof} We first prove that $\overline{v}_i$, $i=1,...,m$, span $V^1_\mu$. For every  $v$ in $V^1_\mu =  (\mbA_1/\mbJ_1) \otimes_{\mbA_1} (V_{\mbA_1})_\mu$ we have 
\begin{align*}
v = \sum\limits_{i=1}^m a_i \otimes_{\mbA_1} v_i = \sum\limits_{i=1}^m b_i (1 \otimes_{\mbA_1} v_i) = \sum\limits_{i=1}^m b_i \overline{v}_i,
\end{align*}
where $a_i = b_i + \mbJ_1 \in \mbA_1/\mbJ_1$, $\overline{v}_i = (1 \otimes_{\mbA_1} v_i)$. So the set $\{1 \otimes_{\mbA_1} v_i\}$ $\C$-spans $ V^1_\mu$. The linear independence follows from the fact that $\{v_i\}$, $i=1,...,m$, are $\mbA_1$-linearly independent. 

We now prove part (b). Assume that $\displaystyle \sum_{i=1}^m c_i(q)v_i = 0$ for $c_i(q) \in \mbA_1$, with some $c_j(q)\neq 0$ for some $j$. Then, multiplying by an appropriate power of $q-1$, we may assume that $c_j(1) \neq 0$. Applying the classical limit gives   $\displaystyle \sum_{i=1}^m c_i(1)\overline{v}_i = 0$.  This contradicts the linear independence of $\overline{v}_i$, $i=1,...,m$.
\end{proof}

The following proposition is the analogue of Propositions \ref{prop:a1wsdecomp} and \ref{prop:a1wsrank} for $V^1$.
\begin{proposition} \label{prop:v1wsdecomp} \hspace{2pt} \\ \vspace{-15pt}
\begin{itemize}
\item[(a)] $V^1 =  \bigoplus\limits_{\mu\leq \lambda} V^1_\mu$
\item[(b)] For each $\mu\in P$, $\dim_{\C}V_\mu^1 = \rank_{\mbA_1}(V_{\mbA_1})_\mu$
\end{itemize}
\end{proposition}

\begin{proof}
The first assertion follows from Proposition \ref{prop:a1wsdecomp}, while the second assertion follows from Lemma \ref{lem:v1wsdecomp}.
\end{proof}

The following theorem shows that the classical limit of $\mfU_q\mfp_n$ is isomorphic to $\mfU\mfp_n$.

\begin{theorem}\label{thm:classicallimit} \hspace{2pt} \\ \vspace{-15pt}
\begin{itemize}
\item[(a)] The elements $\overline{e_i}$, $\overline{e_{\overline{i}}}$, $\overline{f_i}$, $\overline{f_{\overline{i}}}$, ($i\in I$) $\overline{F_{\overline{i}}}$, ($i\in I \cup \{ n\}$) and $\overline{h}$ ($h\in P^\vee$) satisfy the defining relations of $\mfU\mfp_n$. Moreover, there exists a  $\C$-superalgebra isomorphism $\varphi : \mfU\mfp_n \longrightarrow U_1$ and the $U_1$-module $V^1$ has a $\mfU\mfp_n$-module structure.
\item[(b)] For each $\mu \in P$ and $h\in P^\vee$, the element $\overline{h}$ acts on $V_\mu^1$ as scalar multiplication by $\mu(h)$. So, $V_\mu^1$ is the $\mu$-weight space of the $\mfU\mfp_n$-module $V^1$.
\item[(c)] As a $\mfU\mfp_n$-module, $V^1$ is a highest weight module with highest weight $\lambda \in P$ and highest weight vector $\overline{v}$.
\end{itemize}
\end{theorem}

\begin{proof}
To prove (a), we recall that by Theorem 4.1 in \cite{AGG}, there is a  $\C$-superalgebra isomorphism 
\begin{align*}
\psi: \mfU\mfp_n \rightarrow \mfU_{\mcA}\mfp_n/(q-1)\mfU_{\mcA}\mfp_n,
\end{align*} where $\mcA$ is the localization of $\C[q,q^{-1}]$ at the ideal generated by $q-1$, and $\mfU_{\mcA}\mfp_n$ is the $\mcA$-subalgebra of $\mfU_q\mfp_n$ generated by a set of elements $\tau_{ij}$ (for the precise definition of $\tau_{ij}$, see \S4 in \cite{AGG}). The map $\theta: \mfU_{\mcA}\mfp_n/(q-1)\mfU_{\mcA}\mfp_n \rightarrow \mfU_{\mbA_1}/ \mbA_1\mfU_{\mbA_1}$, defined by
\begin{align*}
&\theta(\overline{\tau}_{-i,-i-1}) \mapsto \overline{e_i}  && \theta(\overline{\tau}_{-i,i+1}) \mapsto \overline{e_{\overline{i}}}\\
&\theta(\overline{\tau}_{i,i+1}) \mapsto \overline{f_i} && \theta(\overline{\tau}_{i,-i-1}) \mapsto \overline{f_{\overline{i}}}\\
&\theta(\overline{\tau}_{i,-i}) \mapsto \overline{F_{\overline{i}}}  && \theta(\overline{\tau}_{ii}) \mapsto \overline{k}_i
\end{align*}
is also a  $\C$-superalgebra isomorphism.  Indeed, this follows from the definition of the corresponding generators and the fact that the classical limit of $t_{ij}$ and $\tau_{ij}$ coincide.  Also, we have already established that $\sigma: \mfU_{\mbA_1}/\mbA_1\mfU_{\mbA_1} \rightarrow U_1$ is a $\C$-superalgebra isomorphism by the definition of $U_1$. Therefore, the map
\begin{align*}
 \phi = \sigma\circ\theta\circ\psi: \mfU\mfp_n \longrightarrow U_1
\end{align*} is a  $\C$-superalgebra isomorphism.

To prove (b), let $w\in (V_{\mbA_1})_\mu$ and $h\in P^\vee$. Then we have that
\begin{align*}
(q^h; 0)_qw = \frac{q^h-1}{q-1}w = \frac{q^{\mu(h)}-1}{q-1}w.
\end{align*}
Taking the classical limit of both sides gives $\overline{h}\overline{w} = \mu(h)\overline{w}$, as desired.

It remains to prove (c). From  (b) we have that $\overline{h}\overline{v} = \lambda(h)\overline{v}$. Since $v$ is highest vector of $V_{\mbA_1}$ with highest weight $\lambda$, it follows that $\overline{e_{{i}}} \overline{v} = 0$ and $\overline{e_{\overline{i}}} \overline{v} = 0$ for each $i\in I$. Therefore, (a) and Proposition \ref{prop:highestweightV} imply that $V^1 = U_1^-\overline{v} = \mfU\mfp_n^-\overline{v}$. Hence, by definition, $V^1$ is a highest weight $\mfU\mfp_n$-module with highest weight $\lambda$ and highest weight vector $\overline{v}$. \end{proof}

\begin{proposition}\label{prop:charclasslimit} $\ch V^1 = \ch V^q$ \end{proposition}
\begin{proof}
Propositions \ref{prop:a1wsrank} and \ref{prop:v1wsdecomp} imply  that $\dim_\C V^1_\mu = \dim_{\C(q)} V^q_\mu$ for each $\mu\in P$. This, along with Theorem \ref{thm:classicallimit}(b), gives us the desired result.
\end{proof}

\begin{corollary} $V^q(\lambda)$ is finite dimensional if and only if $\lambda\in \Lambda^+$. \end{corollary}

\begin{proof}
Let $V^q = V^q(\lambda)$ with highest weight vector $v$, and suppose that $\lambda\in \Lambda^+$. From Proposition \ref{prop:irrhighweightfinitedim}, we have that $f_i^{\lambda(k_i)-\lambda(k_{i+1}) + 1}v = 0$ for all $i\in I$. Applying the classical limit leads to $\overline{f}_i^{\lambda(k_i)-\lambda(k_{i+1}) + 1}\overline{v} = 0$. Since $V^1$ is a highest weight module, Proposition \ref{prop:irrhighweightfinitedimclassical} gives us that $V^1$ is finite dimensional. Thus, by  Proposition \ref{prop:charclasslimit},  $V^q$ is finite dimensional.

Conversely, suppose that $V^q(\lambda)$ is finite dimensional. By Proposition \ref{prop:charclasslimit}, we have that $V^1$ is also finite dimensional. By Proposition \ref{thm:dominant-weight-fin-dim-mod}, we have that $\lambda\in \Lambda^+$. \end{proof}

\section{Category of Tensor Representations of $\mfU_q\mfp_n$-modules}\label{sec:rensor-representations}

In this final section we discuss the category of tensor representations of $\mfU_q\mfp_n$.  It is shown in \cite{Mo} that the  $\mfp_n$-module $\C(n|n)^{\otimes k}$ is not completely reducible for any $k\geq 2$. We will prove a similar result for $\mfU_q\mfp_n$.

Let $V = \C_q(n|n)$. The action of the generators $t_{ij}$ of $\mfU_q\mfp_n$ on $V$  is given by the following formulas obtained in  \cite{AGG}:
\begin{eqnarray*}
t_{ii}(u_a) &=& \sum_{b=-n}^n q^{\delta_{bi}(1-2p(i)) + \delta_{b,-i}(2p(i)-1)} E_{bb}(u_a); \\
t_{i,-i}(u_a) &=& (q-q^{-1}) \delta_{i>0} E_{-i,i}(u_a); \\
t_{ij}(u_a) & = &  (q-q^{-1}) (-1)^{p(i)} \msE_{ji}(u_a), \text{ if  }|i|\neq |j|.
\end{eqnarray*}
The action of the Drinfeld-Jimbo generators of $\mfU_q\mfp_n$ in (\ref{Uqpngen}) then follows. We can then extend this action to $V^{\otimes k}$ through comultiplication given in Lemma \ref{lem:comult}. 

Recall also from \cite{AGG}  the $\mfU_q\mfp_n$-module homomorphisms $\mfc: \C_q(n|n)^{\otimes 2} \longrightarrow \C_q(n|n)^{\otimes 2}$ and $\mft: \C_q(n|n)^{\otimes 2} \longrightarrow \C_q(n|n)^{\otimes 2}$, where
\begin{align*}
\mfc =  {} & \sum_{a,b=-n}^n (-1)^{p(a)p(b)} E_{ab} \otimes E_{-a,-b}, \\ 
\mft = {} &  \sum_{i,j=-n}^n (-1)^{p(j)} E_{ij} \ot E_{ji} 
 + (q-1) \sum_{i=1}^n \left(  E_{-i,i} \ot E_{i,-i} \right)  \\
& + (q-1) \sum_{i=1}^n \left(  E_{ii} \ot E_{ii} \right)  
 - (q^{-1}-1) \sum_{i=1}^n  \left(  E_{i,-i} \ot E_{-i,i} \right)  \\
& - (q^{-1}-1) \sum_{i=1}^n  \left( E_{-i,-i} \ot E_{-i,-i} \right)  
 + (q-q^{-1}) \sum_{i=1}^{n} \left( E_{ii} \ot E_{-i,-i} \right)   \\
& + (q-q^{-1}) \sum_{|j|<|i|} \left(E_{jj} \ot E_{ii} \right) 
 + (q-q^{-1}) \sum_{|j|<|i|} \left( (-1)^{p(i)p(j)} E_{ji} \ot E_{-j,-i} \right) .
\end{align*}
These maps are then extended to $\mfU_q\mfp_n$-module homomorphisms  $\mfc_i: \C_q(n|n)^{\otimes k} \longrightarrow \C_q(n|n)^{\otimes k}$ and $\mft_i: \C_q(n|n)^{\otimes k} \longrightarrow \C_q(n|n)^{\otimes k}$ by applying $\mfc$ and $\mft$, respectively, to the $i^{th}$ and $(i+1)^{th}$ tensors. We will refer to the map $\mfc$ as the \textit{contraction} map.

\begin{remark} \label{rem-moon}The contraction map in \cite{Mo} requires a sign change. As a result some of the subsequent theorems have to be modified in \cite{Mo}. Namely, the results in Sections 6.1 and 6.2 in \cite{Mo} need to be corrected and the correct version can be obtained by taking $q=1$  of our results in this section. 

\end{remark}
\subsection{Maximal vectors in $V^{\otimes k}$}
In this subsection we describe the complete set of linearly independent maximal vectors in the $\mfU_q\mfp_n$-module $V^{\otimes k}$. Following the ideas in Theorem 3.8 of \cite{Mo}, we will use a $q$-analogue of the Young symmetrizer defined in \cite{Gy} to define these maximal vectors.

Let $S_k$ be the symmetric group on the set $\{1,\hdots,k\}$, and let $S := \{\mss_1,\hdots, \mss_k\}$ where $\mss_i = (i, i+1)$. Recall that the periplectic $q$-Brauer algebra $\mfB_{q,k}$ is generated by $\mst_i$ and $\msc_i$ for $1\leq i\leq k-1$ and satisfies a set of relations listed in Definition 5.1 of \cite{AGG}. The action of $\mfB_{q,k}$ on  $\C_q(n|n)^{\otimes k}$ is given by  $\mst_i$  and $\msc_i$  acting by $\mft_i$ and $\mfc_i$, respectively.

We consider the Hecke algebra $H_k$ as the subalgebra of $\mfB_{q,k}$  generated by $\{h(\mss_i) = \mst_i  \: | \; i = 1,2,\hdots, k-1 \}$ subject to the following relations:
\begin{align*}
&(h(\mss_i) - q)(h(\mss_i) + q^{-1}) = 0,\\
&h(\mss_i)h(\mss_{i+1})h(\mss_i) = h(\mss_{i+1})h(\mss_i)h(\mss_{i+1}).
\end{align*}
If $\sigma$ is a permutation having a reduced decomposition  $\sigma = \mss_{i_1} \cdots \mss_{i_\ell}$ we set $h(\sigma) = h(\mss_{i_1})\cdots h(\mss_{i_\ell})$. Then $h(\sigma)h(\sigma^\prime) = h(\sigma \sigma^\prime)$  if $\ell(\sigma \sigma^\prime) = \ell(\sigma) + \ell(\sigma^\prime)$, where $\sigma,\sigma^\prime \in S_k$ and $\ell(\sigma)$ is the length of the permutation $\sigma$.

Define the following element of $\mfB_{q,k}$: $$\msc_{r,s} \coloneqq h(\sigma_{r,s})\msc_1 h^{-1}(\sigma_{r,s}),$$ where $\sigma_{r,s} \coloneqq (1,r)(2,s)$.

For  $\displaystyle j \in \left\{ 1,\hdots,\left\lfloor \frac{k}{2} \right\rfloor \right\}$ and  two disjoint ordered subsets  $\tilde{r} =\{r_1, \hdots , r_j\}$ and $\tilde{s} =\{s_1, \hdots , s_j\}$ of $\{1,\hdots , k\}$ such that $r_i < s_i$ for all $i=1,\hdots,j$, set $$\msc_{\tilde{r},\tilde{s}} := \msc_{r_1,s_1}\cdots \msc_{r_j,s_j}, \; \; \; \ \msc_{\emptyset,\emptyset} := \id.$$  We set $(\tilde{r},\tilde{s}) = \{(r_1,s_1),\hdots,(r_j,s_j)\}$ and denote by $\mcP(j)$ the set of all  $(\tilde{r},\tilde{s})$ such that the cardinality of both $\tilde{r}$ and $\tilde{s}$ equal $j$. Set $\displaystyle \mcP=\bigcup_{j=0}^{\left\lfloor \frac{k}{2} \right\rfloor} \mcP(j)$.

We follow the common definition of standard tableau as given for instance in \cite{Gy}. If $\lambda$ is a partition of $N$ we write $\lambda \vdash N$.  Let $\lambda  \vdash N$  have length at most $2n$. Following \cite{Gy}, we define two standard tableaux $T_+ = T_+(\lambda)$ and $T_-=T_-(\lambda)$ depending on $\lambda$, where the entries of $T_+$ increase by one across the rows from left to right, and the entries of $T_-$ increase by one down the columns. Let $W_+$ (respectively, $W_-$) be the group of all elements in $S_N$ which permute the entries within each row of $T_+$ (respectively, each column of $T_-$).

Let
\begin{align*}
\mse_+ = \mse_+(\lambda) &\coloneqq  \sum\limits_{\sigma\in W_+}  q^{\ell(\sigma)}h(\sigma),\\
\mse_- = \mse_-(\lambda) &\coloneqq  \sum\limits_{\sigma\in W_-} (-q)^{-\ell(\sigma)}h(\sigma).
\end{align*}
Note that for each $\mss = \mss_i$ in $W_+$, 
\begin{align*}
\mse_+ &= \sum\limits_{\substack{\sigma\in W_+ \\ \mss\sigma > \sigma}}  q^{\ell(\sigma)}(1+qh(\mss))h(\sigma).
\end{align*}
Thus, we have that $(1-q^{-1}h(\mss))\mse_+ = 0$, or in other words, $h(\mss)\mse_+ = q\mse_+$. Hence $$h(\rho)\mse_+ = q^{\ell(\rho)}\,\mse_+$$ 
 for $\rho\in W_+$. With the same  reasoning we obtain analogous identities included  in  the  following Lemma.
\begin{lemma} \label{lem:q-symmetrizer}
For $\rho\in W_+$ and $\rho^\prime\in W_-$,
\begin{align*}
h(\rho)\mse_+ &= \mse_+h(\rho) = q^{\ell(\rho)}\mse_+,\\
h(\rho^\prime)\mse_- &= \mse_-h(\rho^\prime) = (-q)^{-\ell(\rho^\prime)}\mse_-.
\end{align*}
\end{lemma}

Let $T$ be a standard tableaux of shape $\lambda \vdash k$. Denote by $\sigma_{\pm}^T$  the permutation that transforms $T_{\pm}$ to $T$. 
Note that the set $\sigma_{+}^T W_+ (\sigma_{+}^T)^{-1}$ (respectively, $\sigma_{-}^T W_- (\sigma_{-}^T)^{-1}$) consists of all permutations that permute the rows (respectively, columns) of  $T$.

Define $x_T(q)\in H_k$ by
\begin{align*}
x_T(q) \coloneqq h(\sigma_{-}^T)\, \mse_-\left(h(\sigma_{-}^T)\right)^{-1}  h(\sigma_{+}^T)\, \mse_+  \left(h(\sigma_{+}^T)\right)^{-1}.
\end{align*}
Note that there exists $\xi\in\C(q)$, depending on the shape $\lambda$ of the tableau $T$, such that
\begin{align*}
x_T(q)^2 = \xi x_T(q).
\end{align*}
Then the $q$-analogue of the Young symmetrizer as defined in \cite{Gy} is 
\begin{align*}
y_T(q) = \frac{1}{\xi} x_T(q).
\end{align*}

In what follows for a subset $A$ of $\{1,2,...,k \}$ by $A^c$ we denote its complement.
Denote by $\msS\msT((\tilde{r} \cup \tilde{s})^c)$ the set of all standard tableaux of  shape $\mu \vdash k - 2j$, for some $\displaystyle j \in \left\{ 0,1,\hdots,\left\lfloor \frac{k}{2} \right\rfloor \right\}$, with entries in $(\tilde{r} \cup \tilde{s})^c$, where $(\tilde{r}, \tilde{s})\in \mcP(j)$.

For each $\tau\in \msS\msT((\tilde{r} \cup \tilde{s})^c)$, define the associated simple tensor  of $\tau$ by $w_{\tau,\tilde{r},\tilde{s}} := w_1\otimes \hdots \otimes w_k$, where
\begin{align*}
w_i := \begin{cases} u_1 & \text{ if } i\in\tilde{r}\\u_{-1} & \text{ if } i\in\tilde{s}\\ u_j & \text{ if } j\in(\tilde{r}\cup \tilde{s})^c \text{ and } i \text{ is in } j\text{th row of } \tau.
\end{cases}
\end{align*}

We now prove a $q$-anologue of Theorem 3.8 in \cite{Mo}.

\begin{theorem}\label{thm-maximalvectors} Let $n$ and $k$ be positive integers such that $n\geq k$. Then
\begin{align*}
\{y_{\tau}\msc_{\tilde{r},\tilde{s}}w_{\tau,\tilde{r},\tilde{s}} \mid (\tilde{r},\tilde{s})\in \mcP, \, \tau \in \msS\msT((\tilde{r} \cup \tilde{s})^c), \ell(\tau)\leq n\}
\end{align*}
is a linearly independent set of maximal vectors in the $\mfU_q\mfp_n$-module $V^{\otimes k}$.

\end{theorem}

\begin{proof} 
Let $w = w_{\tau,\tilde{r},\tilde{s}}$ and $\theta = y_{\tau}\msc_{\tilde{r},\tilde{s}}w$. Note that the weight of $\theta$ is the same as the weight of $w$ since $y_{\tau}c_{\tilde{r},\tilde{s}}$ commutes with the action of $\mfU_q\mfp_n$. Note that the fact that $\theta \neq 0$ and the linear independence property follow by applying the classical limit and using Theorem 3.8 in \cite{Mo}.

To show that $\theta$ is a maximal vector, it suffices to show that $\theta$ is annihilated by each root vector $e_{i}$ and $e_{\overline{i}}$, $i\in I$. The action of $e_{i}$ and $e_{\overline{i}}$ on $w$ can be explicitly written as follows
\begin{align}
\begin{split} \label{explict-uqpn-action}
e_{\overline{i}}(w_1\otimes\hdots\otimes w_k) &= \sum\limits_{a=1}^k (-1)^{p(w_1)+\hdots+p(w_{a-1})} q^{k_i}w_1\otimes\hdots\otimes q^{k_i}w_{a-1}\otimes e_{\overline{i}}w_{a}\otimes q^{k_{i+1}}w_{a+1}\otimes\hdots\otimes q^{k_{i+1}}w_k,\\
e_{i}(w_1\otimes\hdots\otimes w_k) &= \sum\limits_{a=1}^k  q^{k_i}w_1\otimes\hdots\otimes q^{k_i}w_{a-1}\otimes e_{i}w_{a}\otimes q^{k_{i+1}}w_{a+1}\otimes\hdots\otimes q^{k_{i+1}}w_k\\
- &\frac{q-q^{-1}}{2}\sum\limits_{\substack{a,b=1 \\ a < b}}^k (-1)^{p(w_1)+\hdots+p(w_{a-1})} q^{k_i}w_1\otimes\hdots\otimes q^{k_i}w_{a-1}\otimes e_{\overline{i}}w_{a}\otimes q^{k_{i+1}}w_{a+1}\otimes\hdots \\&\otimes q^{k_{i+1}}w_{b-1}\otimes F_{\overline{i+1}}w_{b}\otimes q^{k_{i+1}}w_{b+1}\otimes ...\otimes  q^{k_{i+1}}w_k.
\end{split} 
\end{align}

Let $x = e_i$ or $x= e_{\overline{i}}$. We  look at how  $x$ acts on $\msc_{\tilde{r},\tilde{s}}w$.  Note that $\msc V^{\otimes 2}$ is the trivial 
$\mfU_q\mfp_n$-module  $ \C(q)$. Therefore the sum of the terms with $x$ acting on a pair of contracted tensor factors is zero. Also, from \eqref{explict-uqpn-action}, we see that the action of $x$ on non-contracted tensor factors, or in other words where $w_i\in V_0$, those specific terms in the summation will be zero, except when $x=e_{j-1}$ for some $j\geq 2$. Thus $x$ either annihilates $\msc_{\tilde{r},\tilde{s}}w$ or produces a sum of tensors which are obtained by applying $x = e_{j-1}$, for some $j\geq 2$, to a factor unaffected by $\msc_{\tilde{r},\tilde{s}}$. Each of those tensors has one factor whose subscript is in $(\tilde{r} \cup \tilde{s})^c$ and has been lowered by one, so $u_j$ has been changed to $u_{j-1}$. Denote such a tensor by $\msv$. We want to show that $y_\tau\msv = 0$. 

Fix $\psi \in W_+$. There are two factors of the tensor $h(\sigma_{+}^\tau)h(\psi)h(\sigma_{+}^\tau)^{-1}\msv$ that have the same $u_j$. Then there exists a transposition $(a, b) = (\sigma_{-}^\tau) \rho\, (\sigma_{-}^\tau)^{-1} \in  (\sigma_{-}^\tau)W_-(\sigma_{-}^\tau)^{-1}$ which permutes these two factors. Using that $\ell (a,b) = \ell (\rho)$, we obtain
 $$ h(\sigma_{-}^\tau)\,h(\rho)\,h(\sigma_{-}^\tau)^{-1}\,h(\sigma_{+}^\tau)h(\psi)h(\sigma_{+}^\tau)^{-1}\msv = h(a,b)\,h(\sigma_{+}^\tau)h(\psi)h(\sigma_{+}^\tau)^{-1}\msv = q^{\ell(\rho)} h(\sigma_{+}^\tau)h(\psi)h(\sigma_{+}^\tau)^{-1}\msv.$$ 

By Lemma \ref{lem:q-symmetrizer}, we have that for $\rho\in W_-$,
\begin{align*}
(-q)^{-\ell(\rho)}\mse_- = \mse_- \, h(\rho) = \sum\limits_{\sigma\in W_-} (-q)^{-\ell(\sigma)}h(\sigma)\,h(\rho) 
\end{align*}
and hence
\begin{align*}
h(\sigma_{-}^\tau)\mse_-\,h(\sigma_{-}^\tau)^{-1} = \sum\limits_{\sigma\in W_-} (-q)^{-\ell(\sigma)+\ell(\rho)}h(\sigma_{-}^\tau) \, h(\sigma)\, h(\sigma_{-}^\tau)^{-1}\, h(\sigma_{-}^\tau) \,h(\rho)\, h(\sigma_{-}^\tau)^{-1}
\end{align*}
Thus we have
\begin{align*}
h(\sigma_{-}^\tau)\mse_-\,h(\sigma_{-}^\tau)^{-1}h(\sigma_{+}^\tau)h(\psi)h(\sigma_{+}^\tau)^{-1}\msv &= \sum\limits_{\sigma\in W_-} (-q)^{-\ell(\sigma)+\ell(\rho)}h(\sigma_{-}^\tau) \, h(\sigma)\, h(\sigma_{-}^\tau)^{-1}\, h(\sigma_{-}^\tau) \,h(\rho)\, h(\sigma_{-}^\tau)^{-1}h(\sigma_{+}^\tau)h(\psi)h(\sigma_{+}^\tau)^{-1}\msv\\
&= \sum\limits_{\sigma\in W_-} (-q)^{-\ell(\sigma)+\ell(\rho)}q^{\ell(\rho)} h(\sigma_{-}^\tau) \, h(\sigma)\, h(\sigma_{-}^\tau)^{-1}\, h(\sigma_{+}^\tau)h(\psi)h(\sigma_{+}^\tau)^{-1}\msv\\
&=  (-1)^{\ell(\rho)}q^{2\ell(\rho)}\sum\limits_{\sigma\in W_-}  (-q)^{-\ell(\sigma)} h(\sigma_{-}^\tau) \, h(\sigma)\, h(\sigma_{-}^\tau)^{-1}\,h(\sigma_{+}^\tau)h(\psi)h(\sigma_{+}^\tau)^{-1}\msv\\
&= (-1)^{\ell(\rho)}q^{2\ell(\rho)}h(\sigma_{-}^\tau)\mse_-\,h(\sigma_{-}^\tau)^{-1}h(\sigma_{+}^\tau)h(\psi)h(\sigma_{+}^\tau)^{-1}\msv.
\end{align*}
Hence,  $$h(\sigma_{-}^\tau)\mse_-\,h(\sigma_{-}^\tau)^{-1}h(\sigma_{+}^\tau)h(\psi)h(\sigma_{+}^\tau)^{-1}\msv = 0$$ for each $\psi \in W_+$. Therefore, we have that
\begin{align*}
 h(\sigma_{-}^\tau)\mse_-\,h(\sigma_{-}^\tau)^{-1}h(\sigma_{+}^\tau)\mse_+\,h(\sigma_{+}^\tau)^{-1}\msv = \sum\limits_{\psi\in W_+} q^{\ell(\psi)}h(\sigma_{-}^\tau)\mse_-\,h(\sigma_{-}^\tau)^{-1}h(\sigma_{+}^\tau)h(\psi)h(\sigma_{+}^\tau)^{-1}\msv = 0.
\end{align*} This concludes the proof of $y_\tau\msv = 0$. \end{proof}

\subsection{Decomposition of $V^{\otimes 2}$}

By Theorem \ref{thm-maximalvectors}, the following vectors are linearly independent maximal vectors of the $\mfU_q\mfp_n$-module $V^{\otimes 2}$:
\begin{align*}
\ytableausetup{smalltableaux}
\theta_1^q &= y_{\ytableaushort{12}}(u_1\otimes u_1) = u_1\otimes u_1,\\
\theta_2^q &= y_{\ytableaushort{1,2}}(u_1\otimes u_2) = \frac{1}{1+q^{-2}}(q^{-1}u_1\otimes u_2 - u_2\otimes u_1),\\
\theta_3^q &= \msc_1(u_1\otimes u_{-1}) = \sum_{i=-n}^n u_i \otimes u_{-i},
\end{align*}
where 
\begin{align*}
\ytableausetup{smalltableaux}
 y_{\ytableaushort{12}} &= \frac{1}{1+q^2}(1 + q\mst_1),\\
  y_{\ytableaushort{1,2}} &= \frac{1}{1+q^{-2}}(1 - q^{-1}\mst_1).\\
\end{align*}

\begin{proposition} \label{prop-max}
The vectors $\theta_1^q, \theta_2^q, \theta_3^q$ form a complete set of linearly independent maximal vectors of the $\mfU_q\mfp_n$-module $V^{\otimes 2}$. Moreover,  $\theta_1^q, \theta_3^q \in y_{\ytableaushort{12}}V^{\otimes 2}$ and  $\theta_2^q \in  y_{\ytableaushort{1,2}}V^{\otimes 2}$.

\end{proposition}
\begin{proof}
Applying the classical limit to $\theta^q_i$ yields the linearly independent maximal vectors $\theta_i$ in \cite{Mo} (after appropriate sign change for $\theta^q_3$, see Remark \ref{rem-moon}). Since the complete set of linearly independent maximal vectors of the $\mfU\mfp_n$-module $V^{\otimes 2}$ has exactly three vectors, the result follows from Proposition \ref{prop:charclasslimit}. The second statement is subject to a direct verification.
\end{proof}

In what follows we will show that $V^{\otimes 2} $ is isomorphic to the direct sum of the two indecomposable representations, $y_{\ytableaushort{12}}V^{\otimes 2}$ and $ y_{\ytableaushort{1,2}}V^{\otimes 2}$.

\begin{proposition} \label{y1-2-V2-red-indecomposable} The module $y_{\ytableaushort{1,2}}V^{\otimes 2}$ is reducible and indecomposable. More precisely, there is a non-split exact sequence of  $\mfU_q \mfp_n$-modules 
\begin{equation} \label{eq-seq1}
0 \longrightarrow V^q(\epsilon_1 + \epsilon_2)\longrightarrow y_{\ytableaushort{1,2}}V^{\otimes 2} \longrightarrow V^q(0) \longrightarrow 0.
\end{equation}
 \end{proposition}

\begin{proof}
We have that $$\msc_1 y_{\ytableaushort{1,2}}V^{\otimes 2} \subset \msc_1 V^{\otimes 2}\cong \C(q).$$ However, since \begin{align} \label{contractionimage}
\msc_1 y_{\ytableaushort{1,2}} (u_1\otimes u_{-1}) = \frac{1}{1+q^{-2}} \msc_1[q^{-2}u_1\otimes u_{-1} - u_{-1}\otimes u_{-1}] = \theta_3^q \neq 0,
\end{align} it follows that $\msc_1 y_{\ytableaushort{1,2}}V^{\otimes 2}$  cannot be zero. Therefore, $\msc_1 y_{\ytableaushort{1,2}}V^{\otimes 2} \cong \C(q)$. Consider the restriction of the contraction map $\msc_1$ to $y_{\ytableaushort{1,2}}V^{\otimes 2}$, and let $\mcN$ denote the kernel of this restriction. Then  $$y_{\ytableaushort{1,2}}V^{\otimes 2}/\mcN \cong \C(q) \cong V^q(0).$$ 

By Proposition \ref{prop-max}, $\theta_2^q \in \mcN$ and $\theta_1^q, \theta_3^q \notin \mcN$. Hence, by the same proposition,  $\theta_2^q$ is the only, up to a scalar multiple, maximal vector in $\mcN$, in particular, $\mcN$ is simple. Thus $$V^q(\epsilon_1 + \epsilon_2) \cong \mfU_q\mfp_n\theta_2^q = \mcN \subsetneq y_{\ytableaushort{1,2}}V^{\otimes 2}. $$ This implies the exact sequence \eqref{eq-seq1}. The sequence does not split because $y_{\ytableaushort{1,2}}V^{\otimes 2} $ has a unique up to a scalar multiple maximal vector.
\end{proof}

\begin{proposition} \label{y12-V2-red-indecomposable} The module $y_{\ytableaushort{12}}V^{\otimes 2}$ is reducible and indecomposable. More precisely, there is a non-split short exact sequences of $\mfU_q \mfp_n$-modules 
\begin{equation} \label{eq-seq2}
0 \longrightarrow V^q(0) \longrightarrow y_{\ytableaushort{12}}V^{\otimes 2} \longrightarrow V^q(2\epsilon_1) \longrightarrow 0.
\end{equation}

\end{proposition}

\begin{proof}
We have that $$y_{\ytableaushort{12}}\msc_1 = \msc_1.$$ This implies that $$\msc_1 V^{\otimes 2} \subset y_{\ytableaushort{12}}V^{\otimes 2}.$$  Note that $\theta_1^q \not\in \msc_1 V^{\otimes 2}$. However, \eqref{contractionimage} gives that $\theta_3^q\in \msc_1 (V^{\otimes 2})$. By Proposition \ref{prop-max},  $\theta_3^q$ is the only, up to a scalar multiple, maximal vector in $\msc_1 V^{\otimes 2} $. Thus, we have that $$V^q(0) \cong \mfU_q\mfp_n \theta_3^q \subset \msc_1 V^{\otimes 2} \subsetneq y_{\ytableaushort{12}}V^{\otimes 2}.$$ By Proposition \ref{prop-max}, we have $$y_{\ytableaushort{12}}V^{\otimes 2}/\mfU_q\mfp_n \theta_3^q \cong V^q(2\epsilon_1), $$  which implies the exact sequence \eqref{eq-seq2}. Since $\theta_1^q$ generates $y_{\ytableaushort{12}}V^{\otimes 2}$, the sequence does not split.  \end{proof}

The following theorem is the main result of this subsection.

\begin{theorem} \label{uqpn-2tensorproductdecomposition}  As a $\mfU_q\mfp_n$-module, we have the following decomposition
\begin{align*}
\ytableausetup
{smalltableaux}
V^{\otimes 2} = y_{\ytableaushort{12}}V^{\otimes 2} \oplus  y_{\ytableaushort{1,2}}V^{\otimes 2},
\end{align*}
where the submodules in the above decomposition are involved in the non-split short exact sequences \eqref{eq-seq1}, \eqref{eq-seq2}. 
\end{theorem}

\begin{proof}
Let $T: V^{\otimes 2} \longrightarrow  V^{\otimes 2}$ be the $\mfU_q\mfp_n$-module homomorphism defined by $$T(v) = y_{\ytableaushort{12}}v.$$  Note that $$y_{\ytableaushort{12}}y_{\ytableaushort{1,2}} = \frac{1}{(1+q^2)(1+q^{-2})}(1+qt_1)(1-q^{-1}t_1) = 0.$$ So, we have that $ y_{\ytableaushort{1,2}}V^{\otimes 2} \subset \ker T$, which implies that  $ y_{\ytableaushort{1,2}}V^{\otimes 2} = \ker T$ since the only maximal vector of $V^{\otimes 2}$ in $\ker T$ is $\theta_2^q$. Thus we have a short exact sequence
\begin{align*}
0 \longrightarrow y_{\ytableaushort{1,2}}V^{\otimes 2} \longrightarrow V^{\otimes 2}\longrightarrow  y_{\ytableaushort{12}}V^{\otimes 2} \longrightarrow 0.
\end{align*}
Using the embedding $\iota: y_{\ytableaushort{12}} V^{\otimes 2} \longrightarrow  V^{\otimes 2}$ and that $T  \circ  \iota= \id$, we see that the sequence above splits. The remaining part of the theorem follows from Propositions \ref{y1-2-V2-red-indecomposable} and \ref{y12-V2-red-indecomposable}.
\end{proof}

\subsection{Decomposition of $V^{\otimes 3}$}

In this subsection we prove an analogous statement to Theorem \ref{uqpn-2tensorproductdecomposition} for $V^{\otimes 3}$. More precisely, we show that

\begin{align} \label{uqpn-3tensorproductmodule}
\ytableausetup
{smalltableaux}
V^{\otimes 3} = y_{\ytableaushort{123}}V^{\otimes 3} \oplus  y_{\ytableaushort{12,3}}V^{\otimes 3} \oplus  y_{\ytableaushort{13,2}}V^{\otimes 3} \oplus  y_{\ytableaushort{1,2,3}}V^{\otimes 3}
\end{align}
is the decomposition of $V^{\otimes 3}$ into indecomposables, where 
\begin{align*}
y_{\ytableaushort{123}} &= \frac{1}{1+2q^2+2q^4+q^6} (1 + q\mst_1+q\mst_2+q^2\mst_1\mst_2+q^2\mst_2\mst_1+q^3\mst_1\mst_2\mst_1),   \\
y_{\ytableaushort{12,3}} &= \frac{1}{q^{-2}+1+q^2}[1+q\mst_1+(q-q^{-1})\mst_2-\mst_1\mst_2+(q^2-1)\mst_2\mst_1-q\mst_1\mst_2\mst_1], \\
y_{\ytableaushort{13,2}} &=  \frac{1}{q^{-2}+1+q^2} (1-q^{-1}\mst_1-q^2\mst_2\mst_1+q\mst_1\mst_2\mst_1),\\
y_{\ytableaushort{1,2,3}} &= \frac{1}{1+2q^{-2}+2q^{-4}+q^{-6}} (1 - q^{-1}\mst_1-q^{-1}\mst_2+q^{-2}\mst_1\mst_2+q^{-2}\mst_2\mst_1-q^{-3}\mst_1\mst_2\mst_1). 
\end{align*}

With a slight of notation, Theorem \ref{thm-maximalvectors} implies that the following are linearly independent maximal vectors of $V^{\otimes 3}$:
\begin{align*}
\theta_1^q &= \msc_1(u_1\otimes u_{-1} \otimes u_1), \\
\theta_2^q &= \mst_2\msc_1\msc_2(u_1\otimes u_1 \otimes u_{-1}),\\
\theta_3^q &= \msc_2(u_1\otimes u_1 \otimes u_{-1}),\\
\theta_4^q &= y_{\ytableaushort{123}}(u_1\otimes u_1 \otimes u_1), \\
\theta_5^q &= y_{\ytableaushort{12,3}}(u_1\otimes u_1 \otimes u_2), \\
\theta_6^q &= y_{\ytableaushort{13,2}} (u_1\otimes u_2 \otimes u_1), \\
\theta_7^q &= y_{\ytableaushort{1,2,3}} (u_1\otimes u_2 \otimes u_3). \\
\end{align*}

The next proposition is proven in a similar way as Proposition \ref{prop-max}.
\begin{proposition}  \label{prop-max-3}The  vectors $\theta_i^q$, $i=1,...,7$, form a complete set of linearly independent maximal vectors of the $\mfU_q\mfp_n$-module $V^{\otimes 3}$. Furthermore,
\begin{align*}
\theta_4^q, \hspace{10pt} \theta_1^q + q\theta_2^q + q^2\theta_3^q &\in  y_{\ytableaushort{123}}V^{\otimes 3},\\
\theta_5^q, \hspace{10pt} -\theta_1^q - (q-q^{-1})\theta_2^q + \theta_3^q  &\in  y_{\ytableaushort{12,3}}V^{\otimes 3},\\
\theta_6^q, \hspace{10pt} - q\theta_2^q + \theta_3^q &\in y_{\ytableaushort{13,2}}V^{\otimes 3},\\
\theta_7^q &\in y_{\ytableaushort{1,2,3}}V^{\otimes 3}.
\end{align*}
\end{proposition}

We will now look into each of the submodules in the decomposition of (\ref{uqpn-3tensorproductmodule}).

\begin{proposition} \label{prop-y1/2/3} The module $y_{\ytableaushort{1,2,3}}V^{\otimes 3}$ is reducible and indecomposable. More precisely, we have the following non-split short exact sequence
\begin{equation} \label{eq-ne3-1}
0 \longrightarrow V^q(\epsilon_1 + \epsilon_2 + \epsilon_3)\longrightarrow  y_{\ytableaushort{1,2,3}}V^{\otimes 3} \longrightarrow V^q(\epsilon_1) \longrightarrow 0.
\end{equation}
\end{proposition}

\begin{proof}

Consider the contraction map $\msc_1$. Let $\mcN$ be the kernel of the restriction of $\msc_1$ to $y_{\ytableaushort{1,2,3}} V^{\otimes 3} $. We have that $$\msc_1 y_{\ytableaushort{1,2,3}} V^{\otimes 3} \subset \msc_1 V^{\otimes 3} \cong V = \C_q(n|n)$$ and $$\msc_1y_{\ytableaushort{1,2,3}}(u_1\otimes u_{-1}\otimes u_1) =  \frac{1}{1+2q^{-2}+2q^{-4}+q^{-6}}\msc_1[(1+q^2-q^{-4})u_1\otimes u_{-1}\otimes u_1 - q^{-2}u_1\otimes u_1\otimes u_{-1} + u_{-1}\otimes u_1\otimes u_1] \neq 0.$$ By Proposition \ref{prop-max-3}, $\theta_7^q$ is the only maximal vector in $\mcN$. Hence,  $$V^q(\epsilon_1 + \epsilon_2 + \epsilon_3)\cong \mfU_q\mfp_n\theta_7^q = \mcN \subsetneq  y_{\ytableaushort{1,2,3}} V^{\otimes 3}.$$
Moreover, $$y_{\ytableaushort{1,2,3}} V^{\otimes 3} / \mfU_q\mfp_n\theta_7^q \cong V^q(\epsilon_1)\cong \C_q(n|n),$$
which implies the exact sequence. The sequence does not split because $ y_{\ytableaushort{1,2,3}}V^{\otimes 3} $ has a unique up to a scalar multiple maximal vector.
\end{proof}

\begin{proposition} \label{prop-y123} The module $y_{\ytableaushort{123}}V^{\otimes 3}$ is reducible and indecomposable. More precisely, we have the following non-split short exact sequence
\begin{equation} \label{eq-ne3-2}
0 \longrightarrow V^q(\epsilon_1) \longrightarrow  y_{\ytableaushort{123}}V^{\otimes 3} \longrightarrow V^q(3\epsilon_1) \longrightarrow 0.
\end{equation}
\end{proposition}

\begin{proof}
Let $\msK = \msc_1\msc_2+q\mst_2\msc_1\msc_2 +q^2\msc_2$. Note that $$y_{\ytableaushort{123}}\, \msK= \msK.$$ It follows that $$ \msK V^{\otimes 3} \subset y_{\ytableaushort{123}} V^{\otimes 3}.$$ Note that  $\theta_4^q \neq \msK V^{\otimes 3}$. However,  $\msK(u_1\otimes u_1 \otimes u_{-1}) = \theta_1^q + q\theta_2^q + q^2\theta_3^q$. By Proposition \ref{prop-max-3}, $\theta_1^q + q\theta_2^q + q^2\theta_3^q$ is the only, up to a scalar multiple, maximal vector in $\msK V^{\otimes 3} $. Thus, we have that $$V^q(\epsilon_1) \cong \mfU_q\mfp_n(\theta_1^q + q\theta_2^q + q^2\theta_3^q) \subsetneq y_{\ytableaushort{123}} V^{\otimes 3}. $$ By Proposition \ref{prop-max-3}, we have $$y_{\ytableaushort{1,2,3}} V^{\otimes 3} / \mfU_q\mfp_n(\theta_1^q + q\theta_2^q + q^2\theta_3^q) \cong V^q(3 \epsilon_1).$$ which implies the exact sequence \eqref{eq-ne3-2}. Since $\theta_4^q$ generates $y_{\ytableaushort{12}}V^{\otimes 2}$, the sequence does not split. 
\end{proof}

\begin{proposition} \label{prop-y12/3} The module $y_{\ytableaushort{12,3}}V^{\otimes 3}$ is completely reducible as a $\mfU_q\mfp_n$-module into a direct sum of irreducible $\mfU_q\mfp_n$-modules. More precisely, we have the following split short exact sequence
\begin{equation} \label{eq-ne3-3}
0 \longrightarrow V^q(2\epsilon_1 + \epsilon_2)\longrightarrow  y_{\ytableaushort{12,3}}V^{\otimes 3} \longrightarrow V^q(\epsilon_1) \longrightarrow 0.
\end{equation}
\end{proposition}

\begin{proof}
Consider the $\mfU_q\mfp_n$-module homomorphism $T: y_{\ytableaushort{12,3}}V^{\otimes 3}  \rightarrow \msc_2V^{\otimes 3}$ such that $T(y_{\ytableaushort{12,3}}v) = \msc_2y_{\ytableaushort{12,3}}v$. Since $$\msc_2 y_{\ytableaushort{12,3}}\msc_2 = \frac{q^{-2}}{q^{2}+1+q^{-2}}\msc_2,$$ we have that $T$ is surjective and $$\msc_2y_{\ytableaushort{12,3}}V^{\otimes 3} = \msc_2V^{\otimes 3}.$$ 
Note that $$y_{\ytableaushort{12,3}}\msc_2(u_1\otimes u_1\otimes u_{-1}) = -\theta_1^q - (q-q^{-1})\theta_2^q + \theta_3^q. $$ Thus, $ -\theta_1^q - (q-q^{-1})\theta_2^q + \theta_3^q\in y_{\ytableaushort{12,3}}\msc_2V^{\otimes 3}$. However,  $u_1\otimes u_1\otimes  u_2 \not\in \msc_2V^{\otimes 3}$, which implies that $\theta^q_5\in y_{\ytableaushort{12,3}}\msc_2V^{\otimes 3}$. So, by Proposition \ref{prop-max-3}, we have that $$ V^q(\epsilon_1)\cong \mfU_q\mfp_n\big(-\theta_1^q - (q-q^{-1})\theta_2^q + \theta_3^q\big) = y_{\ytableaushort{12,3}}\msc_2V^{\otimes 3} .$$ Since $\msc_2V^{\otimes 3} \cong V^q(\epsilon_1)$, we have an inclusion map $i: \msc_2V^{\otimes 3} \rightarrow y_{\ytableaushort{12,3}}V^{\otimes 3} $ such that $i\circ T = \text{id}$. Therefore, the following short exact sequence splits:
\begin{align*}
0 \longrightarrow \ker T \longrightarrow y_{\ytableaushort{12,3}}V^{\otimes 3} \xlongrightarrow{T} \msc_2V^{\otimes 3} \longrightarrow 0
\end{align*}
By Proposition \ref{prop-max-3}, since $\theta_5^q\in \ker T$, we have that $\ker T = \mfU_q\mfp_n\theta_5^q \cong V^q(2\epsilon_1 + \epsilon_2)$. This implies the exact sequence \eqref{eq-ne3-3}.
\end{proof}

\begin{proposition} \label{prop-y13/2} The module $y_{\ytableaushort{13,2}}V^{\otimes 3}$ is isomorphic to $y_{\ytableaushort{12,3}}V^{\otimes 3}$ as $\mfU_q\mfp_n$-modules.  More precisely, we have the following split short exact sequence
\begin{equation} \label{eq-ne3-4}
0 \longrightarrow V^q(2\epsilon_1 + \epsilon_2)\longrightarrow  y_{\ytableaushort{13,2}}V^{\otimes 3} \longrightarrow V^q(\epsilon_1) \longrightarrow 0.
\end{equation}

\end{proposition}

\begin{proof} Let $v_1 = \theta_6^q$, $w_1 = \theta_5^q$, $v_2 =  -q\theta_2^q + \theta_3^q $, and $w_2 =  -\theta_1^q - (q-q^{-1})\theta_2^q + \theta_3^q$. 
Consider the $\mfU_q\mfp_n$-module homomorphism $$S: y_{\ytableaushort{12,3}}V^{\otimes 3}  \longrightarrow y_{\ytableaushort{13,2}}V^{\otimes 3} $$ such that $$S(w_1) = v_1\hspace{20pt} \text{ and } \hspace{20pt} S(w_2) = v_2.$$ By Propositions \ref{prop-max-3} and  \ref{prop-y12/3}, $S$ is an isomorphism. The short exact sequence \eqref{eq-ne3-4} then follows from Proposition \ref{prop-y12/3}. \end{proof}

The following theorem is analogous to Theorem \ref{uqpn-2tensorproductdecomposition}, but in the case of $V^{\otimes 3}$.
\begin{theorem} \label{uqpn-3tensorproductdecomposition}  As a $\mfU_q\mfp_n$-module, we have the following decomposition

\begin{align} \label{eq-v3-decomposition}
V^{\otimes 3} = y_{\ytableaushort{123}}V^{\otimes 3} \oplus  y_{\ytableaushort{12,3}}V^{\otimes 3} \oplus  y_{\ytableaushort{13,2}}V^{\otimes 3} \oplus  y_{\ytableaushort{1,2,3}}V^{\otimes 3},
\end{align}
where each submodule in the above decomposition are involved in either the non-split short exact sequences \eqref{eq-ne3-1} and \eqref{eq-ne3-2}, or in the split short exact sequences \eqref{eq-ne3-3} and \eqref{eq-ne3-4}.
\end{theorem}
\begin{proof} 
Let $T: V^{\otimes 3} \longrightarrow  V^{\otimes 3}$ be the $\mfU_q\mfp_n$-module homomorphism defined by $$T(v) = y_{\ytableaushort{123}}v.$$  Note that $$y_{\ytableaushort{123}}\msK =  0,$$ where $\msK = y_{\ytableaushort{12,3}} , y_{\ytableaushort{13,2}},  y_{\ytableaushort{1,2,3}} $. By Proposition \ref{prop-max-3}, this implies that $-\theta_1^q - (q-q^{-1})\theta_2^q + \theta_3^q , -q\theta_2^q + \theta_3^q, \theta_q^5, \theta_q^6, \theta_q^7\in \ker T$. We have the short exact sequence
\begin{align*}
0 \longrightarrow \ker T \longrightarrow V^{\otimes 3}\longrightarrow  y_{\ytableaushort{123}}V^{\otimes 3} \longrightarrow 0.
\end{align*}
Using the embedding $\iota: y_{\ytableaushort{123}} V^{\otimes 3} \longrightarrow  V^{\otimes 3}$ and that $T  \circ  \iota= \id$, we see that the sequence above splits. 

Let $T^\prime: \ker T \longrightarrow  \ker T$ be the $\mfU_q\mfp_n$-module homomorphism defined by $$T(v) = y_{\ytableaushort{12,3}}v.$$  Note that $$y_{\ytableaushort{12,3}}\msK =  0,$$ where $\msK = y_{\ytableaushort{13,2}},  y_{\ytableaushort{1,2,3}} $. By Proposition \ref{prop-max-3}, this implies that $ -q\theta_2^q + \theta_3^q, \theta_q^6, \theta_q^7\in \ker T^\prime$. Using the embedding $\iota^\prime: y_{\ytableaushort{12,3}} V^{\otimes 3} \longrightarrow  V^{\otimes 3}$ and that $T^\prime  \circ  \iota^\prime = \id$, we see that the short exact sequence \begin{align*}
0 \longrightarrow \ker T^\prime \longrightarrow \ker T \longrightarrow  y_{\ytableaushort{12, 3}}V^{\otimes 3} \longrightarrow 0.
\end{align*}
 splits. Using similar arguments, we have that $\ker T^\prime = y_{\ytableaushort{13,2}}V^{\otimes 3} \oplus y_{\ytableaushort{1,2,3}} V^{\otimes 3}$, and thus the decomposition \eqref{eq-v3-decomposition} follows.

The remaining part of the theorem follows from Propositions \ref{prop-y1/2/3}, \ref{prop-y123}, \ref{prop-y12/3}, and \ref{prop-y13/2}.
\end{proof}

\subsection{Reducibility of $V^{\otimes k}$}

\begin{theorem} \label{thm:vk-not-reducible}
For every $k\geq 2$, the $\mfU_q\mfp_n$-module $V^{\otimes k}$ is not completely reducible.
\end{theorem}
\begin{proof} Fix $k\geq 2$, and assume for the sake of contradiction that  $V^{\otimes k}$  is completely reducible. For any $r,s \in \{1,2\hdots, k-1\}$ such that $r\neq s$, $$\msc_{r,s}V^{\otimes k}\cong V^{\otimes k-2}.$$  Consecutive applications of $\msc_{r,s}$ to  $V^{\otimes k}$  for appropriate $r$ and $s$ will lead to a submodule $M$ of $V^{\otimes k}$  that is isomorphic either to $V^{\otimes 2}$ or to $V^{\otimes 3}$. By Theorems \ref{uqpn-2tensorproductdecomposition} and \ref{uqpn-3tensorproductdecomposition},  $V^{\otimes 2}$ and $V^{\otimes 3}$ are not completely reducible. This  leads to a contradiction as submodules of completely reducible modules are also completely reducible.
\end{proof}

\noindent
Department of Mathematics, University of Texas at Arlington 
\\ Arlington, TX 76021, USA\\
saber.ahmed\@@mavs.uta.edu

\noindent
Department of Mathematics, University of Texas at Arlington 
\\ Arlington, TX 76021, USA\\
grandim\@@uta.edu

\noindent
University of Alberta, Department of Mathematical and Statistical Sciences, CAB 632\\
Edmonton, AB T6G 2G1, Canada\\
nguay\@@ualberta.ca

\end{document}